\RequirePackage{fix-cm}
\documentclass{amsart}       

\usepackage{ccaption}
\usepackage[font={it}]{caption}

\usepackage{amsaddr}
\usepackage{tikz}
\usepackage{amsmath}
\usepackage{amsfonts}
\usepackage{amssymb}
\usepackage{amsthm}
\usepackage{epstopdf}
\usepackage{graphicx}
\usepackage{enumerate}
\usepackage{xfrac}
\usepackage{amstext}    
\usepackage{array} 
\usepackage[margin=4cm]{geometry}
\usepackage{hyperref}

\usepackage{bm}

\usepackage{todonotes}

\newcommand{\ab}{\mathbf{a}}
\newcommand{\pb}{\mathbf{p}}
\newcommand{\kb}{\mathbf{k}}
\newcommand{\lb}{\mathbf{l}}
\newcommand{\jb}{\mathbf{j}}

\newcommand{\rb}{\mathbf{r}}

\newcommand{\vb}{\mathbf{v}}
\newcommand{\bb}{\mathbf{b}}
\newcommand{\cb}{\mathbf{c}}
\newcommand{\eb}{\mathbf{e}}
\newcommand{\Gb}{\mathbf{\Gamma}}
\newcommand{\aspect}{\kappa}
\newcommand{\xb}{\mathbf{x}}

\newcommand{\linco}{\boldsymbol{\mu}}
\newcommand{\omb}{\boldsymbol{\omega}}
\newcommand{\Z}{\mathbb{Z}}

\newcommand{\aspb}{{ \bm{\kappa}} }
\newcommand{\aspm}{{ K} }
\newcommand{\transp}{\mathsf{T}}
\usepackage{scalerel,stackengine}
\stackMath
\newcommand\reallywidehat[1]{%
\savestack{\tmpbox}{\stretchto{%
  \scaleto{%
    \scalerel*[\widthof{\ensuremath{#1}}]{\kern-.6pt\bigwedge\kern-.6pt}%
    {\rule[-\textheight/2]{1ex}{\textheight}}
  }{\textheight}%
}{0.5ex}}%
\stackon[1pt]{#1}{\tmpbox}%
}

\newcommand{\crossmat}[1]
{\widehat{#1}}

\numberwithin{equation}{section}
\numberwithin{figure}{section}

\usetikzlibrary{patterns}

\theoremstyle{plain}
          \newtheorem{theorem}{Theorem}[section]
          \newtheorem{proposition}{Proposition}[section]
          \newtheorem{lemma}[theorem]{Lemma}

          \theoremstyle{definition}
          \newtheorem{definition}[theorem]{Definition}

\newcommand{\rem}[1]{}

\begin{document}

\title{Stability Theory of the 3-Dimensional Euler Equations}


\author{Holger R.~Dullin$^{(1)}$ and Joachim Worthington$^{(1),(2)}$}
\address[foot]{ (1) School of Mathematics and Statistics, The University of Sydney, 
NSW, Australia\\ 
			(2) Cancer Research Division, Cancer Council NSW, NSW, Australia}

\begin{abstract}

The Euler equations on a three-dimensional periodic domain have a family of  shear flow steady states. We show that the linearised system around these steady states decomposes into subsystems equivalent to the linearisation of shear flows in a two-dimensional periodic domain. To do so, we derive a  formulation of the dynamics of the vorticity Fourier modes on a periodic domain and linearise around the shear flows. The linearised system has a decomposition analogous to the two-dimensional problem, which can be significantly simplified. By appealing to previous results it is shown that some subset of the shear flows are spectrally stable, and another subset are spectrally unstable.  For a dense set of parameter values the linearised operator has a nilpotent part, leading to linear instability. This is connected to the nonnormality of the linearised dynamics and the transition to turbulence.
Finally we show that all shear flows in the family considered (even the linearly stable ones) are parametrically unstable.

\keywords{Euler Equations \and Hamiltonian Dynamics \and Hydrodynamics}
\end{abstract}

\maketitle

\section{Introduction}

The Euler equations on a three-dimensional domain present many  unanswered questions, despite centuries of research, see, e.g., \cite{Gibbon08} for a 
recent review. Questions like the possibility of finite-time blowup which are known for the two-dimensional domain either do not hold or are open problems in the three-dimensional domain. 
In this paper, we study steady states of the Euler equations with velocity vector field $V$  and vorticity $\Omega$ of the form
\begin{equation} \label{eqn:OmegaStar}
\quad  V^* =\frac{2\Gb\times\pb}{|\pb|^2} \sin ( \pb\cdot \mathbf{x}  ),
\quad \Omega^*=2\Gb \cos (\pb \cdot \mathbf{x}  ), 
\quad \xb \in \mathbb{T}^3, \quad \pb \in \mathbb{Z}^3, \quad \Gb \in \mathbb{R}^3
\end{equation}
on the periodic domain $\mathbb{T}^3$. 
For a recent study of these so called  Kolmogorov flows see, e.g., \cite{Shebalin97}.
These are shear flows with a somewhat unusual direction relative to the periodic domain. 
The study of the stability of shear flows goes back to Rayleigh \cite{Rayleigh79},  who 
introduced the stability criterion named after hime. 
More recently improvements of such criteria were given in 
\cite{Morrison13} and \cite{Morrison98}.
Our methods are not as general and somewhat specific to the particular 
shear flows under consideration. 
We show that for values of $\pb$  satisfying $|\pb| > \sqrt{3} + 3/2$
these steady states are unstable. This is achieved by linearisation, decomposition, and reduction to an equivalent two-dimensional case.
For certain $\pb$ parallel to a unit vector linear stability is possible, but then the value of $\Gb$ becomes important.
Even though there are linearly stable shear flows in our family, they are parametrically unstable.
Parametric instability means that the equilibrium may be linearly stable, 
but can be made linearly unstable by an arbitrarily small perturbation of $\Gb$. 
The reason for the prevalence of parametric instability is that even though the set of $\Gb$ that leads to instability has zero measure, 
it is dense in the set of allowed $\Gb$ values. These results also hold for periodic domain with unequal side lengths.

As a particular example that highlights the fundamental difference between the two- and three dimensional domains consider 
the flow with velocity field $V^*(\xb) = \sin x (0, 1)^\transp$ for on a two-dimensional torus with size $\xb=(x,y)\in \left [ 0,{2\pi} \right ) \times \left [ 0,\frac{2\pi}{\aspect_y} \right )$ for $\aspect_y\in\mathbb{R}$ satisfying $\aspect_y>  1$. According to a classical result by 
Arnold \cite{Arnold78} this steady state is Lyapunov stable as shown by an energy-Casimir argument.
On a three-dimensional domain $\xb=(x,y,z)\in \left [ 0,{2\pi} \right ) \times \left [ 0,\frac{2\pi}{\aspect_y} \right )\times \left [ 0,\frac{2\pi}{\aspect_z} \right )$ for $\aspect_y,\aspect_z>1$, 
the spectrum of the linearisation about $V^*(\xb) = \sin x  (0, \varphi, 1)^\transp $ is on the imaginary axis for any $\varphi \in \mathbb{R}$, 
so like the system in two-dimensions it is spectrally stable.
However, for rational $\varphi$ the steady state is linearly unstable, because the linearised operator has  nilpotent blocks.
When $\varphi$ is irrational the steady state is linearly stable but parametrically unstable
since perturbing $\varphi$ into a nearby rational number  gives a linearly unstable steady state.
This shows that in three dimensions stability and instability are so intricately intertwined  
that the concept of linear stability does not capture the essence of the physical behaviour.
We show that the linear operator is highly nonnormal, which provides a mechanism for 
the nonlinear loss of stability by finite perturbations.

\section{The Euler Equations and Shear Flows}

We begin by deriving a formulation of the Euler equations in terms of the Fourier coefficients of the vorticity. This relies on the divergence-free property which restricts the values the vorticity can take. 
We then introduce three-dimensional parallel shear flows.
Linearising around these shear flows allows for a decomposition into classes (block-diagonalisation) as in two dimensions, 
even though the governing differential equations are significantly more complex than those for the two-dimensional equations (see e.g. \cite{Li00},\cite{Dullin19}).

\subsection{Vorticity Formulation of the three-dimensional Euler Equations}

\label{sec:vortform}

The Euler equations for an incompressible, inviscid flow on the domain $\mathcal{D}\subset \mathbb{R}^3$ are
\begin{equation}
\Omega=\nabla\times V 
\label{eq:omvrelation}
\end{equation}
\begin{equation}
\frac{\partial \Omega}{\partial t}=(\Omega \cdot \nabla)V-(V \cdot \nabla ) \Omega.
\label{eq:omvpde}
\end{equation}
 where $V(\xb,t):\mathcal{D}\times \mathbb{R} \to\mathbb{R}^3$  is the velocity field  and $\Omega(\xb,t):\mathcal{D}\times \mathbb{R} \to\mathbb{R}^3$ is the vorticity field.  See Constantin \cite{Constantin07} for a recent discussion of these equations.
The velocity and vorticity also satisfy the divergence free conditions
\begin{equation}
	\nabla \cdot V=0,\;\; \nabla \cdot \Omega=0.
\end{equation}
For the velocity, this is due to  the incompressibility. For the vorticity, this follows from the relationship $\Omega=\nabla\times V$ and the fact that the divergence of the curl must be zero.

We will consider these equations on the domain 
\begin{equation}
	\mathcal{D}=[ 0,{2\pi} )\times[ 0,{2\pi} )\times[ 0,{2\pi} )
\end{equation}
with periodic boundary conditions in every direction on $V$ and $\Omega$ so $V(0,y,z)=V(2\pi,y,z)$, etc.
This is the isotropic periodic domain with equal lengths in each dimension. In Section \ref{sec:aniso} we will extend this to a general three-dimensional rectangular periodic domain.

Expand $V$ and $\Omega$ into Fourier series. Define the Fourier coefficients with wavenumber $\jb\in\mathbb{Z}^3$ by 
\begin{equation}
	\vb_\jb(t)=\int_\mathcal{D} V(\xb,t)e^{-i\jb\cdot\xb}\mathrm{d}\xb,
\end{equation}
\begin{equation}
	\omb_\jb(t)=\int_\mathcal{D} \Omega(\xb,t)e^{-i\jb\cdot\xb}\mathrm{d}\xb
\end{equation}
where $\xb\in\mathcal{D}$ is the spatial variable.
Then
\begin{equation}
	\label{eq:3Dfourier}
	\begin{split}
	V=\sum_{\jb\in\mathbb{Z}^3} \vb_\jb(t)e^{i\jb\cdot\xb},\\
	\Omega=\sum_{\jb\in\mathbb{Z}^3} \omb_\jb(t)e^{i\jb\cdot\xb}.
	\end{split}
\end{equation}
As $V$ and $\Omega$ are real, $\vb_{-\jb}=\bar{\vb}_\jb$ and $\omb_{-\jb}=\bar{\omb}_\jb$. 

The divergence free conditions imply
\begin{equation}
	\label{eq:divfree}
	\jb\cdot\vb_\jb=0,\;\; \jb\cdot\omb_\jb=0.
\end{equation}
All dynamics occur in the \emph{divergence-free subspace}
\begin{equation}
\label{eq:divfreess}
	\{ \omb_\jb \; | \; \jb\cdot\omb_\jb=0 \}.
\end{equation}

We wish to write the system as a set of differential equations for the dynamics of the vorticity modes $\omb_\jb$ only. 
The condition $\Omega=\nabla \times V$ implies
\begin{equation}
\label{eq:omvv}
	\omb_\jb=i\jb\times \vb_\jb.
\end{equation}

Combining \eqref{eq:divfree} and \eqref{eq:omvv} implies
\begin{equation}
	\jb\times\omb_\jb=-i |\jb|^2\vb_\jb.
\end{equation}
Thus for all $\jb\neq \mathbf{0}$,
we can take the inverse to the curl on the divergence-free subspace
\begin{equation}
	\label{eq:velocitymodes}
	\vb_\jb=i\frac{\jb \times\omb_\jb}{|\jb|^2}.
\end{equation}
This allows us to invert the relationship \eqref{eq:omvrelation} and write formally $V=(\nabla\times)^{-1}\Omega$.
Now
\begin{equation} \begin{split}
	(\Omega\cdot\nabla)V & =\sum_{\jb,\kb\in\mathbb{Z}^3}i\jb\cdot\omb_\kb \vb_\jb e^{i(\jb+\kb)\cdot\xb} \\
		&=-\sum_{\jb,\kb\in\mathbb{Z}^3}  \jb\cdot\omb_{\kb}
		\frac{\jb\times\omb_\jb}{|\jb|^2}  e^{i(\jb+\kb)\cdot\xb}
\end{split}\end{equation} 
and 
\begin{equation} \begin{split}
	(V\cdot\nabla)\Omega& =\sum_{\jb,\kb\in\mathbb{Z}^3}i\jb\cdot\vb_{\kb}
		\omb_\jb e^{i(\jb+\kb)\cdot\xb } \\ 
		&=-\sum_{\jb,\kb\in\mathbb{Z}^3} \frac{1}{|\kb|^2}\jb\cdot(\kb\times\omb_\kb)
			\omb_\jb e^{i(\jb+\kb)\cdot\xb}.
\end{split}\end{equation} 

Rewriting the partial differential equation $\Omega_t=(\Omega \cdot \nabla)V-(V \cdot \nabla ) \Omega$ in Fourier space leads to infinitely many ordinary differential equations\footnote{See Appendix \ref{app:crosshat} for the definition and some properties of the cross product matrix $\widehat{\jb}$.}
\begin{equation} \begin{split}
\label{eq:3Dvorticityode}
	\dot{\omb}_\jb&
		=\sum_{\kb+\lb=\jb} \frac{1}{|\kb|^2}\left ( \lb\cdot(\kb\times \omb_{\kb}) \omb_{\lb}
		 -\kb\cdot\omb_\lb \kb \times \omb_\kb \right ) \\
		& =\sum_{\kb\in\mathbb{Z}^3\setminus\{\mathbf{0}\}} \left ( \omb_{\jb+\kb}[\kb\times\jb]^{\transp}
		- (\kb\cdot\omb_{\jb+\kb}) \crossmat{\kb}\right ) \frac{\omb_{-\kb}}{|\kb|^2}.
\end{split}\end{equation} 
This is a differential equation for the vorticity modes that does not depend on the velocity.
Defining
\begin{equation}
\label{eq:Aode}
	A(\jb,\kb,\xb):= \left (\xb  \left (\kb\times\jb   \right ) ^{\transp}
		 -(\kb\cdot\xb ) \crossmat{\kb}
		 \right ) 
\end{equation}
for $\jb,\kb,\xb\in\mathbb{R}^3$,
the differential equations can be written as
\begin{equation}
	\label{eq:3dode}
	\dot{\omb_\jb}=\sum_{\kb\in\mathbb{Z}^3\setminus \{ \mathbf{0}\}} A(\jb,\kb, \omb_{\jb+\kb}   )\frac{\omb_{-\kb}}{|\kb|^2}.
\end{equation}
It has been shown that this can be cast as a Poisson system \cite{Dullin18}.

\begin{lemma}
\label{lem:AProp}
For any orthogonal $3\times 3$ matrix $R\in SO(3)$, $A(R\jb, R\kb, R\xb) = R A(\jb,\kb,\xb) R^\transp$.
If $(\jb+\kb)\cdot \xb = 0$, then $\jb\cdot A(\jb,\kb,\xb) = \mathbf{0}$.
\end{lemma}
\begin{proof}
By the properties of the cross product, $R\kb\times R\jb=R(\kb\times \jb)$. Also, $(R\kb)\cdot(R\xb)=\kb^\transp R^\transp R \xb = \kb\cdot\xb$. Finally, $\crossmat{R\kb}=R\crossmat{\kb}R^\transp$ (see Appendix \ref{app:crosshat}). Thus the equivariance follows. 

For the second property observe that
$\jb^\transp A(\jb, \kb, \xb) = (\jb\cdot \xb) ( \kb \times \jb)^\transp  - (\kb\cdot \xb) (\jb \times \kb)^\transp=(\jb+\kb)\cdot \xb (\kb\times\jb)^\transp = 0$. 
\end{proof}
The second property in Lemma~\ref{lem:AProp} corresponds to  the divergence free condition. 
For brevity in the following, we use the notation $A(\jb, \kb):= A(\jb, \kb, \omb_{\jb+\kb})$.

If $\jb=\mathbf{0}$, we observe that
\begin{equation}
\begin{split}
	\dot{\omb}_\mathbf{0}=&\frac{1}{2} i\sum_\kb \bigg [ \left ((\kb\cdot\omb_{-\kb}) \vb_{\kb}-(\kb\cdot\omb_{\kb}) \vb_{-\kb} \right )
	-\left ((\kb\cdot\vb_{-\kb}) \omb_{\kb}-(\kb\cdot\vb_{\kb}) \omb_{-\kb} \right ) \bigg ] 
	=0
\end{split}
\end{equation}
by the reality condition on the modes. Thus the mode $\omb_{\mathbf{0}}$ is constant and
moreover $\omb_{\mathbf{0}}=0$ by \eqref{eq:omvv}, and therefore can be omitted in the summation so there is no singularity in \eqref{eq:3Dvorticityode}.

It is straightforward to check that the dynamics governed by the differential equations \eqref{eq:3Dvorticityode} reduce to the two-dimensional equivalent derived in, e.g., \cite{DMW16} under the assumption that $\omb_\jb=(0,0,\omega_\jb)$ and $\jb=(j_1,j_2,0)$, i.e. ignoring all fluid velocity in the third dimension.

\subsection{Shear Flows}

The equations \eqref{eq:omvrelation}, \eqref{eq:omvpde} admit a family of steady states
\begin{equation}
\label{eq:generalss}
\begin{split}
	V^*&=\boldsymbol{\alpha} f( \pb\cdot\xb ),\\
	\Omega^*&=(\pb\times \boldsymbol{\alpha}) f'(\pb\cdot\xb)
\end{split}
\end{equation}
for any $2\pi$-periodic function $f:\mathbb{R}\to\mathbb{R}$, direction $\pb\in\mathbb{Z}^3$, and vector $\boldsymbol{\alpha}\in\mathbb{R}^3$ satisfying $ \pb\cdot\boldsymbol{\alpha}=0$ for the divergence free condition. 
Then $V=V^*$, $\Omega=\Omega^*$ satisfies \eqref{eq:omvrelation} and is a steady state for \eqref{eq:omvpde}. These shear flows are analogous to those that exist for the two-dimensional domain and are well-studied, see, e.g., \cite{Arnold98,Friedlander99,Li00,DMW16}.

Now consider \eqref{eq:generalss} in the particular case 
\begin{equation}
	f(x)=a\sin(x)+b\cos(x),\;\;\boldsymbol{\alpha}=\frac{2\Gb\times\pb}{|\pb|^2},\;\;\Gb\in\mathbb{R}^3.
\end{equation}
We can simplify this to consider $f(x)=\sin(x)$ by using the identity $a\sin (x)+b\cos (x)=\gamma \sin (x+x_0)$ for appropriate $\gamma$, $x_0$, followed by a rescaling.  
Then the steady state is
\begin{equation}
\begin{split}
	V^*&=\frac{2\Gb\times\pb}{|\pb|^2} \sin ( \pb\cdot \mathbf{x}  ),\\
	\Omega^*&=2\Gb \cos (\pb\cdot \mathbf{x}  ).
\end{split}
\label{eq:sinshear}
\end{equation}

In terms of vorticity Fourier coefficients, \eqref{eq:sinshear} is the equilibrium 
\begin{equation}
	\label{eq:3Dfourierequil}
	\omb^*_\kb=
	\begin{cases}
		\Gb \text{ if } \kb=\pm\pb; \\
		\mathbf{0}\text{ otherwise}
	\end{cases}
\end{equation}
of \eqref{eq:3dode}.
For the divergence-free condition on $\omb_\pb$ to be satisfied, $\Gb$ must  satisfy $\pb\cdot\Gb=0$. 

\subsection{Linearisation}

We now linearise the vector field  \eqref{eq:3dode} around \eqref{eq:3Dfourierequil} to study spectral and linear stability. Denote $\omb_\jb=((\omb_\jb)_x,(\omb_\jb)_y,(\omb_\jb)_z)^T$. We first calculate the Jacobian:
\begin{equation}
	\frac{\partial}{\partial \omb_\kb}\dot{\omb}_\jb =\frac{1}{|\kb|^2}A(\jb,-\kb)
	+\frac{\partial  }{\partial\omb_\kb}   A(\jb,\kb-\jb) \frac{\omb_{\jb-\kb}}{|\jb-\kb|^2}.
\end{equation}
where 
\begin{equation}
	\frac{\partial}{\partial \omb_\kb}\dot{\omb}_\jb=\begin{pmatrix}
		\frac{\partial}{\partial (\omb_\kb)_x} (\dot{\omb}_\jb)_x &
		\frac{\partial}{\partial (\omb_\kb)_x} (\dot{\omb}_\jb)_y &
		\frac{\partial}{\partial (\omb_\kb)_x} (\dot{\omb}_\jb)_z \\
		\frac{\partial}{\partial (\omb_\kb)_y} (\dot{\omb}_\jb)_x &
		\frac{\partial}{\partial (\omb_\kb)_y} (\dot{\omb}_\jb)_y &
		\frac{\partial}{\partial (\omb_\kb)_y} (\dot{\omb}_\jb)_z \\
		\frac{\partial}{\partial (\omb_\kb)_z} (\dot{\omb}_\jb)_x &
		\frac{\partial}{\partial (\omb_\kb)_z} (\dot{\omb}_\jb)_y &
		\frac{\partial}{\partial (\omb_\kb)_z} (\dot{\omb}_\jb)_z 
	\end{pmatrix}
\end{equation}
is the $(\jb,\kb)$ three-by-three block of the Jacobian.
		
Evaluating this at the equilibrium,
\begin{equation}
	\frac{\partial}{\partial \omb_\kb}\dot{\omb}_\jb|_{\Omega^*} =
	\begin{cases}
	\frac{1}{|\jb-\pb|^2} A(\jb,\pb-\jb)|_{\Omega^*}
	+\frac{1}{|\pb|^2}\frac{\partial}{\partial\omb_{\jb-\pb}}(A(\jb,-\pb){\Gb})\;\text{ if }\kb=\jb-\pb; \\
	\frac{1}{|\jb+\pb|^2}A(\jb,-\pb-\jb)|_{\Omega^*}
	+\frac{1}{|\pb|^2}\frac{\partial}{\partial\omb_{\jb+\pb}}(A(\jb,\pb)\Gb)\;\text{ if }\kb=\jb+\pb;\\
	0\;\text{ otherwise}.
	\end{cases}
\end{equation}

Then  the system linearised about $\Omega^*$  written in  $\linco_\jb = \omb_\jb-\omb_\jb^*$ is
\begin{equation}
\label{eq:linsystemCHI}
	\dot{\linco}_\jb=\chi_+(\jb+\pb,\pb,\Gb)\linco_{\jb+\pb}+\chi_-(\jb-\pb,\pb,\Gb)\linco_{\jb-\pb} \,.
\end{equation}
 where 
\begin{equation}
	\chi_\pm(\jb,\pb,\Gb)=\frac{1}{|\jb|^2} A(\jb\mp\pb,-\jb)|_{\Omega^*}	
			+\frac{1}{|\pb|^2}\frac{\partial}{\partial\omb_{\jb}}(A(\jb\mp\pb,\pm\pb){\Gb}).
\end{equation}
Keep in mind that $\pb$ and $\Gb$ are fixed by the choice of steady state.

Now
\begin{equation}
	A(\jb\pm\pb,-\jb)|_{\Omega^*}=\mp\Gb[\jb\times\pb]^{\transp} -(\jb\cdot\Gb) \crossmat{\jb} 
\end{equation} 
\begin{equation}
	\frac{\partial}{\partial\omb_{\jb}}(A(\jb\pm\pb,\mp\pb)\Gb)=
		\mp\Gb\cdot\left ( \pb\times\jb\right ) \mathbb{I}_3-\left [ \pb\times\Gb \right ]  \pb^{\transp} 
\end{equation} 
so $\chi_\pm$ 
can be written as
\begin{equation}
	\chi_\pm(\jb,\pb,\Gb)=\chi_a(\jb,\pb,\Gb) \pm \chi_b(\jb,\pb,\Gb)
\end{equation}
where the terms with constant sign are collected in
\begin{equation} 
\label{eq:chidef1}
	\chi_a(\jb,\pb,\Gb)=-\frac{1}{|\jb|^2} (\jb\cdot\Gb) \, \crossmat{\jb} 
	-\frac{1}{|\pb|^2} \left ( \pb\times\Gb \right ) \pb^{\transp}  
\end{equation} 
and the terms with alternating sign are collected in
\begin{equation}
\label{eq:chidef2}
	\chi_b(\jb,\pb,\Gb)=\frac{1}{|\jb|^2} \Gb(\jb\times\pb)^{\transp}
	+\frac{1}{|\pb|^2}  \Gb\cdot(\pb\times \jb)\mathbb{I}_3
\end{equation} 

\begin{lemma}  \label{lem:chi}
For any orthogonal $3\times 3$ matrix $R$ we have $\chi_\pm(R\jb, R\pb, R\Gb) = R \chi_\pm(\jb,\pb,\Gb) R^\transp$.
The divergence free condition $\jb\cdot \omb_\jb = 0$  for all $\jb \in \mathbb{Z}^3$ implies that $\jb^\transp \chi_\pm(\jb \pm \pb,\pb,\Gb) = \mathbf{0}$.
\end{lemma}
\begin{proof}
These properties are inherited from $A(\cdot,\cdot,\cdot)$, when evaluated at the equilibrium, 
where $A( \jb \pm \pb, -\jb, \omega_{\pm \pb}) = A(\jb \pm \pb, -\jb, \Gb)$.
Moreover, by direct calculation one can verify that $( \partial A(\jb \pm \pb, -\jb, \Gb) / \partial \Gb ) \Gb$
has the same equivariance property. The divergence free property follows directly from that of $A(\cdot,\cdot,\cdot)$.
\end{proof}

\subsection{Class Decomposition of the Linearised System}
\label{sec:classdecomp}

\begin{figure}
\includegraphics[width=0.6\textwidth]{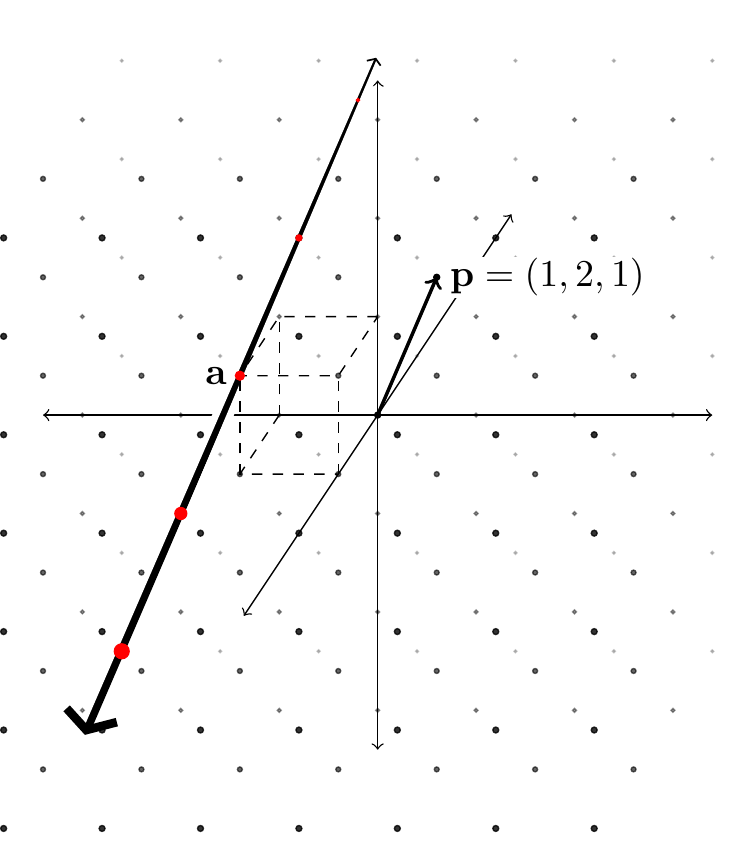}
\caption[The class decomposition of the linearised three-dimensional Euler equations.]{The class decomposition into subsystems for the linearised three-dimensional Euler equations. In this setting, $\pb \in \Z^3$. The lattice points on the line parallel to $\pb$ passing through $\ab$ all belong to one class of modes. The linearised dynamics of the modes in a class only depend on other modes in the same class. For this figure, $\pb=(1,2,1)$ and the class led by $\ab=(-1,1,1)$ is shown.}
\label{fig:3Ddecomp}
\end{figure}

The differential equations \eqref{eq:linsystemCHI} have a clear class decomposition; the dynamics of  mode $\linco_\ab$ depend only on the modes $\linco_{\ab\pm\pb}$, which in turn depend only on $\linco_\ab$ and $\linco_{\ab\pm 2\pb}$, and so on. Thus the modes $\ab+n\pb$ for $n\in \mathbb{Z}$ are a decoupled subsystem for any starting mode number $\ab\in\mathbb{Z}^3$. This is illustrated in Figure \ref{fig:3Ddecomp}, and works much the same as in the two-dimensional problem (see, e.g., Li \cite{Li00}). 

In the linearised system, any mode $\linco_{q\pb}$ for $q\in\mathbb{Q}$ and $q\pb \in \mathbb{Z}^3$ is constant. 
For these modes
\begin{equation}
	\chi_+(q\pb,\pb,\Gb)=\chi_-(q\pb,\pb,\Gb)=-\frac{1}{|\pb|^2} \left ( \pb \times \Gb \right ) \pb^{\transp} .
\end{equation}
So the dynamics are given by 
\begin{equation}
	\dot{\linco}_{q\pb}=-\frac{1}{|\pb|^2}\left ( \pb \times \Gb \right )   [\pb\cdot \linco_{(q+1)\pb}
	+\pb\cdot \linco_{(q-1)\pb}].
\end{equation}
But $\pb\cdot \linco_{m\pb}=\frac{1}{m}(m\pb)\cdot \linco_{m\pb}=0$ by the divergence-free condition, and thus $\dot{\linco}_{q\pb}=0$ for all $q\in\mathbb{Q}$. 

We define the principal domain of mode numbers
\begin{equation}
\label{eq:ADomain3D}
	\mathcal{A}=\{ \ab \in \mathbb{Z}^3 \; | \;
		-\frac{1}{2}|\pb|^2 < \ab\cdot\pb \leq \frac{1}{2} |\pb|^2 \}.
\end{equation}
Then for all $\jb\in\mathbb{Z}^3$, $\jb=\ab+n\pb$ for a \emph{unique} $\ab\in \mathcal{A}$ and $n\in \mathbb{Z}$.

Now when studying the linearised system, we can study each particular class individually. We introduce the following notation: for a fixed $\ab\in \mathcal{A}$, 
and  $\pb, \Gb$ fixed by the equilibrium define
\begin{equation}
	\linco_n:=\linco_{\ab+n\pb},\;\;\chi_\pm(n):=\chi_\pm(\ab+n\pb, \pb, \Gb).
\end{equation}
We have reused the variables $\linco$ and $\chi_i$; the appropriate definition depends on whether we consider the variables a function of an integer $n$ or a vector $\vb$. In the case of an integer we are considering the dynamics of a single class with fixed $\pb$, $\Gb$, $\ab$. 

Then the differential equations for a single class are
\begin{equation}
\label{eq:classODEs3D}
	\dot{\linco}_n=\chi_+(n+1)\linco_{n+1}+\chi_-(n-1)\linco_{n-1}.
\end{equation}

\section{Simplifying the Class Dynamics}

\label{sec:simplifying1}

In the previous section, it was shown that the linearised Euler equations around a shear flow decouple into ``classes'' of modes.
This can be thought of as a block-diagonalisation of the linear operator where each class represents a block, and there are
infinitely many such blocks, as many as there are lattice points in the principle domain $\mathcal{A}$ define in \eqref{eq:ADomain3D}.
In this section, we show that the dynamics of the modes in a class decompose into a subsystem with simple dynamics and another subsystem isomorphic to the dynamics of a corresponding class in the two-dimensional problem. To do so, we transform $\pb$ to a unit vector parallel to the $x$-axis, use the divergence-free property to simplify $\Gb$, and project down to the divergence-free subspace. The dynamics of the resulting system split into a part equivalent to the corresponding two-dimensional problem and secondary dynamics. 
There are exceptional cases when this reduction to the two-dimensional problem is not possible. 
This occurs when $\pb$, $\Gb$ and  $\ab$ are coplanar, and  leads to instability. 
This kind of instability is a truly three-dimensional effect, and is discussed in Section \ref{sec:linclasses}.

\subsection{Simplifying the parameters}
\label{sec:3Drescaling}

As we have shown in Lemma \ref{lem:chi} the matrices $\chi_\pm(n)$ are equivariant 
under rotations. Hence by transforming the variables $\linco_n$ with a fixed rotation 
the matrices $\chi_\pm(n)$ may be simplified. 

\begin{lemma} \label{lem:redpar}
There exists an orthonormal basis $\mathbf{e}_x$, $\mathbf{e}_y$, $\mathbf{e}_z$ in which
\begin{equation}
    \pb = (p, 0, 0), \quad 
    \ab =p(\tilde{a}_x, \tilde{a}_y, 0), \quad
    \Gb = (0, \Gamma\cos\theta, \Gamma\sin\theta)
\end{equation}
where 
\begin{equation}
  \label{eq:3dparam}
  	p=|\pb|,\quad \tilde{a}_x=\frac{ \ab \cdot \pb }{\pb\cdot \pb},\quad  \tilde{a}_y=\frac{|\ab\times\pb|}{\pb\cdot\pb},\quad \Gamma=|\Gb|,
  	\end{equation}
\end{lemma}
and $\theta\in[0,2\pi)$ is determined by 
\begin{equation}
	\cos\theta=\frac{(\Gb\cdot \ab)|\pb|}{\Gamma|\pb\times\ab|},\quad \sin\theta=\frac{\Gb\cdot (\pb\times\ab)}{\Gamma|\pb\times\ab|}.
\end{equation}

\begin{proof}

\begin{figure}
\includegraphics[width=0.4\textwidth]{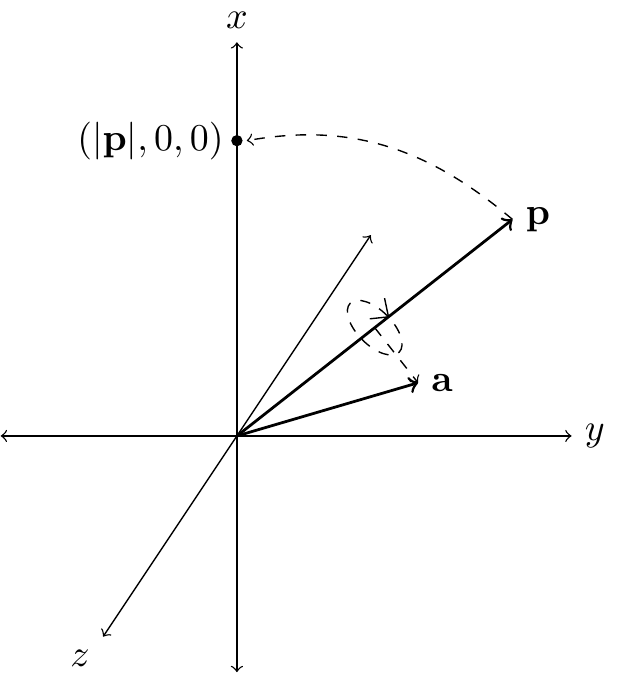}
\includegraphics[width=0.4\textwidth]{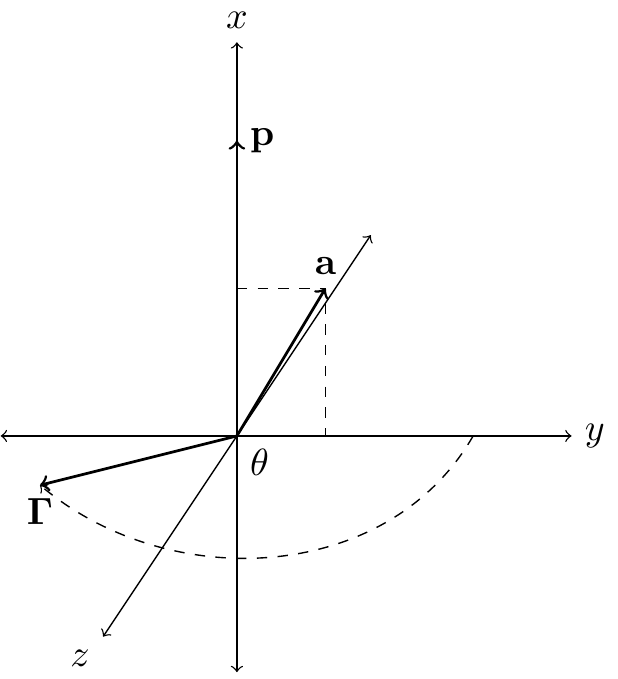}
\caption[Coordinate transformations for the linearised classes in the three-dimensional Euler equations.]{Transformations for the linearised classes. The dynamics for the linearised classes  \eqref{eq:classODEs3D} can be simplified significantly by a change of coordinates. First, the vector $\pb$ is rotated to the $x$ axis. An additional rotation around the $x$-axis aligns $\ab$ with the $x$-$y$ plane. As $\Gb$ is perpendicular to $\pb$, it lies in the $y$-$z$ plane, and we can describe it by its length $|\Gb|$ and angle to the positive $y$-axis $\theta$. The dynamics are invariant under these transformations, so we can consider only classes with such parameters. }
\label{fig:linsimp}
\end{figure}

Consider the vectors 
\begin{equation} \mathbf{v}_x=\pb,\quad \vb_y=\ab-\frac{\ab\cdot \pb}{\pb\cdot \pb}\pb,\quad\vb_z=\pb\times\ab. \end{equation}
These are well-defined and nonzero, as $\pb \neq \mathbf{0}$ and $\pb$ and $\ab$ are not parallel for classes with non-trivial dynamics, by Section \ref{sec:classdecomp}. If $\eb_i=\frac{\vb_i}{|\vb_i|}$ for $i=x,y,z$, $\eb_x,\eb_y,\eb_z$ form an orthonormal basis. These have been chosen so that $\pb$ can be expressed in terms of $\eb_x$ only, and $\ab$ can be expressed in terms of $\eb_x$ and $\eb_y$ only.

We  calculate that $|\vb_x|=|\pb|=p$, $|\vb_y|=\frac{|\pb\times\ab|}{|\pb|}$, and $|\vb_z|=|\pb\times\ab|$. Then $\pb=|\pb|\eb_x$ and $\ab=p\tilde{a}_x\eb_x+p\tilde{a}_y\eb_y$, where $\tilde{a}_x,\tilde{a}_y$ are defined above. Geometrically, $\tilde{a}_x$ is the projection of $\ab$ in the direction of $\pb$ scaled by $|\pb|$, and $\tilde{a}_y$ is the remaining orthogonal component. As $\Gb\cdot \pb=0$, we can write $\Gb=\Gamma(\cos\theta \eb_y+\sin\theta \eb_z)$ where $\theta\in[0,2\pi)$ is determined by the equations given above. 

This change of basis does not alter the dynamics given by the matrices $\chi_\pm(n)$ as we are transforming between orthonormal bases, and $\chi_{\pm}(n)$ are equivariant under rotations by Lemma \ref{lem:chi}.
\end{proof}

Note the exceptional case for $\theta$ when $\Gb \cdot ( \pb \times \ab)=0$. Then $\theta=0$ or $\theta=\pi$.  Geometrically, this is the case where 
$\Gb$, $\pb$ and $\ab$ are coplanar, so that the vector triple product vanishes. 
For the reduction to the two-dimensional case we need to require $\theta \not = 0, \pi$.
Classes with $\theta=0,\pi$ have qualitatively different dynamics to all other classes, which will be discussed in Section \ref{sec:linclasses}.

We now consider each class as a problem with parameters $p,\tilde{a}_x,\tilde{a}_y,\theta\in\mathbb{R}$. We shall see that $p$, $\Gamma$ and $\theta$ have a limited impact on the dynamics of the class, and can map parameter values $(\tilde{a}_x,\tilde{a}_y)$ to parameters of a corresponding class in the linearised two-dimensional Euler equations.

In these reduced parameters there is quite some simplification in the expression for $\chi_\pm$. 
Writing 
\begin{equation}
\alpha_n=\frac{|\ab+n\pb|}{|\pb|}=\sqrt{(\tilde{a}_x+n)^2+\tilde{a}_y^2}
\end{equation}
 we can write
\begin{equation}
\begin{split}
&\chi_a(n)=
\begin{pmatrix}
	0 & 0 & 0 \\
	1 & 0 & 0 \\
	0 & 0 & 0
\end{pmatrix}\Gamma\sin\theta 
+\begin{pmatrix}
			0 & 0 & -\tilde{a}_y^2 \\
		0 & 0 & (\tilde{a}_x+n) \tilde{a}_y \\		
		-(\tilde{a}_x+n)^2 & - (\tilde{a}_x+n) \tilde{a}_y & 0 \end{pmatrix} \frac{1}{\alpha_n^2}\Gamma \cos\theta,	
		\end{split}
\end{equation}
\begin{equation}
	\begin{split}
	\chi_b(n)&=\begin{pmatrix} 1 & 0 & 0 \\
		0 & 1 & 0 \\
		0 & 0 &  1-\frac{1}{\alpha_n^2}  
				   \end{pmatrix}  \tilde{a}_y \Gamma \sin\theta 
+\begin{pmatrix} 0 & 0 & 0 \\
		 0 & 0 & -1 \\
		0 & 0 & 0
				   \end{pmatrix} \frac{ \tilde{a}_y}{\alpha_n^2} \Gamma  \cos\theta\,.
	\end{split}
\end{equation}
Note that this is independent of $p$, and the $\Gamma$ dependence is purely multiplicative. 

\subsection{The Anisotropic Periodic Domain}

\label{sec:aniso}

The previous sections assumed the periodic domain was isotropic for simplicity, so the length of the domain is the same in all three dimensions. We can extend this to the anisotropic periodic domain with size 
\begin{equation}
	\mathcal{D}=\left [ 0,\sfrac{2\pi}{\aspect_x} \right ) \times \left [ 0,\sfrac{2\pi}{\aspect_y} \right ) \times \left [ 0,\sfrac{2\pi}{\aspect_z} \right )
\end{equation}
for $\aspect_x,\aspect_y,\aspect_z\in\mathbb{R}^+$. Write $\aspb:=(\aspect_x,\aspect_y,\aspect_z)$ and define the matrix
\begin{equation}
\aspm:=\begin{pmatrix} \aspect_x & 0 & 0 \\
		0 & \aspect_y & 0 \\
		0 & 0 & \aspect_z 
		\end{pmatrix}.
	\end{equation} 
	
The simplest way to describe the effect of different aspect ratios of the domain size is 
to pass from an integer lattice $\Z^3$ for the wavenumbers to the lattice $\mathcal{L} = \aspect_x \Z \times \aspect_y \Z \times \aspect_z \Z$.

In particular this applies to $\pb$, $\ab$, $\jb$, $\kb$, but not to $\omb$, $\Gb$.

\subsection{Projecting to the divergence-free subspace}

We know that the divergence free subspace \eqref{eq:divfreess} is invariant, so that for initial conditions satisfying $\jb \cdot \omb_{\jb} = 0$ for all $\jb$ this condition holds for all times. 
This is true in both the full dynamics and  in the linearised dynamics.
As the divergence-free subspace is invariant, we project down to this space to find the true dynamics of our system by disallowing perturbations off this subspace.  This projection is achieved by introducing a rotation on the coordinates $\mu_{n,z}$ (where $\linco_n=(\mu_{n,x},\mu_{n,y},\mu_{n,z})$)
which will depend on $n$ and simplify the divergence-free condition.

Let $\rb_x(n)=\frac{\ab+n\pb}{|\ab+n\pb|}=\frac{\tilde{a}_x+n}{\alpha_n}\eb_x+\frac{\tilde{a}_y}{\alpha_n}\eb_y$. This is well defined as $\alpha_n\neq 0$ for $\ab,\pb$ not parallel. Then $\rb_x(n)$ is a unit vector and we can rewrite the divergence-free condition as $\rb_x(n)\cdot \linco_{\ab+n\pb}=0$. Let $\rb_z(n)=\mathbf{e}_z$, and $\rb_y(n)=\rb_x(n)\times\rb_z(n)=\frac{-\tilde{a}_y}{\alpha_n}\eb_x+\frac{\tilde{a}_x+n}{\alpha_n}\eb_y$. Then define 
\begin{equation}
	R(n)=(\rb_x(n),\;\rb_y(n),\;\rb_z(n)).
\end{equation}
It is clear that $R(n)$ is a rotation matrix $R(n)\in SO(3)$ for all $n\in\mathbb{Z}$.

Define $\bar{\linco}_n:=R(n)^{\transp}\linco_n$ and write $\bar{\linco}=(\bar{\mu}_{n,x},\bar{\mu}_{n,y}\bar{\mu}_{n,z})$, so $\bar{\mu}_{n,i}=\rb_{i}\cdot \linco_n$ for $i=x,y,z$.
Then the divergence-free invariance implies that for initial conditions on the divergence-free subspace $\bar{\mu}_{n,x}=\rb_x(n)\cdot \linco_n=0$. 
We thus need only consider the dynamics 
of  $\bar{\mu}_{n,y}$ and $\bar{\mu}_{n,z}$.

\begin{lemma}
Let $S_{\pm}(n)= R(n\mp 1)^{\transp}\chi_\pm(n)R(n)$. Then $S_{\pm}(n)= \begin{pmatrix} *  & 0 \\ * &  \tilde \chi_\pm(n) \end{pmatrix}$
where 
\begin{equation}
		\tilde\chi_\pm(n)=\frac{\Gamma\tilde{a}_y}{\alpha_n^2}\Bigg [
			\begin{pmatrix}
			\pm \alpha_n \alpha_{n\mp1} & 0 \\
			0 & \pm(\alpha_n^2-1)
			\end{pmatrix}\sin\theta
			+
			\begin{pmatrix}
			0 & \alpha_{n\mp1} \\
			0 & 0
			\end{pmatrix}\cos\theta \Bigg ]		
		\,.
\end{equation} 

\end{lemma}
\begin{proof}
The $i$-$j$ entry of $S_\pm(n)$ is
\begin{equation}
    (S_{\pm}(n))_{i,j} = \rb_i(n\mp 1) \cdot \chi_\pm(n) \rb_j(n)
\end{equation}
for $i,j\in \{x,y,z\}$.
We can explicitly calculate $(S_{\pm}(n))_{x,y}=(S_{\pm}(n))_{x,z}=0$. 
Thus $\frac{\mathrm{d}}{\mathrm{d}t}\bar{\mu}_{n,x}=\mathcal{C}_\pm \bar{\mu}_{n,x}$ 
where
$$\mathcal{C}_\pm =\frac{(a_x\mp 1) a_x a_y \Gamma \sin\theta}{\alpha(n\mp 1) \alpha(n)}.$$
and the divergence-free subspace 
$\bar{\mu}_{n,x} = 0$ is invariant as expected. Note, however, that since $|\mathcal{C}_\pm|$ may be 
larger than 1 a numerical scheme based on the equations without restriction to the divergence free
subspace could be unstable.

We can now calculate 
\begin{equation}
\begin{split}
\tilde\chi_\pm(n)
	&=\begin{pmatrix} 
	\rb_y(n\mp 1) \cdot \chi_\pm(n) \rb_y(n) &
	\rb_y(n\mp 1) \cdot \chi_\pm(n) \rb_z(n) \\
	\rb_z(n\mp 1) \cdot \chi_\pm(n) \rb_y(n) &
	\rb_z(n\mp 1) \cdot \chi_\pm(n) \rb_z(n) 
\end{pmatrix}	\\
	&=\frac{\Gamma \tilde{a}_y}{\alpha_n^2}\begin{pmatrix} 
	\pm\alpha_{n\mp 1}\alpha_n \sin \theta &
	\alpha_{n\mp 1}\cos \theta \\
	0 &
	\pm(\alpha_n^2-1) \sin \theta
\end{pmatrix}
\end{split}
\end{equation}
as required.
\end{proof}

The differential equations for the truncated $\tilde{\linco}_n=(\bar{\mu}_{n,y}, \bar{\mu}_{n,z})$ thus are
\begin{equation}
	\dot{\tilde\linco}_n=\tilde \chi_+(n+1) \tilde\linco_{n+1}+\tilde \chi_-(n-1) \tilde\linco_{n-1} \,.
\end{equation}
Note that there is no $p$ dependence, and the $\Gamma \tilde{a}_y$ dependence can be removed by a time rescaling, so the differential equations can be essentially be expressed in terms of $\alpha_n$ and $\theta$. Compare this to the two-dimensional linearised classes in \cite{DMW16,Dullin16} which are written purely in terms of $\rho_n=1-\alpha_n^{-2}$ (after some transformations).

By reordering the coordinates as
\begin{equation}
	\tilde{\linco}=(...,\tilde{z}_{-1,y},\tilde{z}_{0,y},\tilde{z}_{1,y},...,\tilde{z}_{-1,z},\tilde{z}_{0,z},\tilde{z}_{1,z},...),
\end{equation}
the system can be written in block form as 
\begin{equation}
\label{eq:ombode}
\dot{\tilde{\linco}}=\Gamma \tilde{a}_y\mathcal{M}\tilde{\linco},
	\end{equation}
\begin{equation}
\label{eq:3DM1}
\mathcal{M}=\left ( \begin{pmatrix} 
	 M_1 & 0  \\
	0 &  M_2
	\end{pmatrix}\sin\theta+ \begin{pmatrix} 
	0 &  M_3 \\
	0 & 0
	\end{pmatrix}\cos\theta\right )
\end{equation}
where 
 \begin{equation}
	M_1=\begin{pmatrix}
		\ddots & \vdots & \vdots & \vdots & \vdots & \vdots & \ddots \\
		\ldots & 0 & \frac{\alpha_{-2}}{\alpha_{-1}} & 0 & 0 & 0 &\ldots \\
		\ldots & -\frac{\alpha_{-1}}{\alpha_{-2}}  & 0 & \frac{\alpha_{-1}}{\alpha_{0}} & 0 & 0 & \ldots \\
		\ldots &  0 &  -\frac{\alpha_{0}}{\alpha_{-1}} & 0 & \frac{\alpha_{0}}{\alpha_{1}} & 0  & \ldots \\
		\ldots & 0 & 0 &  -\frac{\alpha_{1}}{\alpha_{0}} & 0 & \frac{\alpha_{1}}{\alpha_{2}}&  \ldots \\
		\ldots & 0 & 0 &  0 & -\frac{\alpha_{2}}{\alpha_{1}} & 0 &   \ldots \\
		\ddots & \vdots & \vdots & \vdots & \vdots & \vdots & \ddots 
	\end{pmatrix},
\end{equation}
 \begin{equation}
	M_2=\begin{pmatrix}
		\ddots & \vdots & \vdots & \vdots & \vdots & \vdots & \ddots \\
		\ldots & 0 & \frac{\alpha_{-1}^2-1}{\alpha_{-1}^2} & 0 & 0 & 0 &\ldots \\
		\ldots & \frac{-\alpha_{-2}^2+1}{\alpha_{-2}^2}  & 0 &\frac{\alpha_{0}^2-1}{\alpha_{0}^2} & 0 & 0 & \ldots \\
		\ldots &  0 &   \frac{-\alpha_{-1}^2+1}{\alpha_{-1}^2} & 0 & \frac{\alpha_{1}^2-1}{\alpha_{1}^2} & 0  & \ldots \\
		\ldots & 0 & 0 &   \frac{-\alpha_{0}^2+1}{\alpha_{0}^2} & 0 & \frac{\alpha_{2}^2-1}{\alpha_{2}^2} &  \ldots \\
		\ldots & 0 & 0 &  0 &  \frac{-\alpha_{1}^2+1}{\alpha_{1}^2} & 0 &   \ldots \\
		\ddots & \vdots & \vdots & \vdots & \vdots & \vdots & \ddots 
	\end{pmatrix},
\end{equation}
and
 \begin{equation}
	M_3=\begin{pmatrix}
		\ddots & \vdots & \vdots & \vdots & \vdots & \vdots & \ddots \\
		\ldots & 0 & \frac{\alpha_{-2}}{\alpha_{-1}^2} & 0 & 0 & 0 &\ldots \\
		\ldots & \frac{\alpha_{-1}}{\alpha_{-2}^2}  & 0 & \frac{\alpha_{-1}}{\alpha_{0}^2} & 0 & 0 & \ldots \\
		\ldots &  0 &  \frac{\alpha_{0}}{\alpha_{-1}^2} & 0 & \frac{\alpha_{0}}{\alpha_{1}^2} & 0  & \ldots \\
		\ldots & 0 & 0 &  \frac{\alpha_{1}}{\alpha_{0}^2} & 0 & \frac{\alpha_{1}}{\alpha_{2}^2}&  \ldots \\
		\ldots & 0 & 0 &  0 & \frac{\alpha_{2}}{\alpha_{1}^2} & 0 &   \ldots \\
		\ddots & \vdots & \vdots & \vdots & \vdots & \vdots & \ddots 
	\end{pmatrix}.
\end{equation}

\subsection{Reducing to Two-Dimensional Linearised Class}
\label{sec:laxred}

At this point it is clear that the spectrum of the linearised dynamics  of a class in three dimensions
as determined by $\mathcal{M}$ is the union of the spectra of $\sin \theta M_1$ and $ \sin \theta M_2$. 
We are now going to prove a stronger result, which shows that $\mathcal{M}$ can be 
block-diagonalised when $\theta \not = 0, \pi$.
The first block will have nearly trivial dynamics, while the second block describes dynamics isomorphic to a 
corresponding two-dimensional linearised Euler flow. Define the transformation matrix
\begin{equation}
	T=\text{diag}(...,\alpha_{-2},\alpha_{-1},\alpha_{0},\alpha_{1},\alpha_{2},...).
\end{equation}

As $\ab$ and $\pb$ are not parallel, $\alpha_n\neq 0$ for all $n$, so $T$ is invertible. 
We will discuss the singular case $\theta=0, \pi$ in Section \ref{sec:linclasses} and consider the possible  dynamical implications in Section \ref{sec:nearnil}.
For now assumel $\theta\neq 0, \pi$, so that $\cot (\theta)$ is nonsingular and we can define
\begin{equation}
\label{eq:TTran}
	\mathcal{T}=\begin{pmatrix} T & \cot(\theta) T \\ 0 & \mathbb{I} \end{pmatrix}
\end{equation}
with inverse
\begin{equation}
	\mathcal{T}^{-1}=\begin{pmatrix} T^{-1} & -\cot(\theta) \mathbb{I} \\ 0 & \mathbb{I} \end{pmatrix}.
\end{equation}
Now we conjugate $\mathcal{ M}$ by $\mathcal{T}$ and define
\begin{equation}
\label{eq:goodmeq}
\begin{split}
	\tilde{\mathcal{M}}(\eta)&:= 
		T^{-1}{\mathcal{M}}T 
	= \begin{pmatrix} 
	 \tilde{M}_1 & 0  \\
	0 &  \tilde{M}_2
	\end{pmatrix}\sin\theta+ \begin{pmatrix} 
	0 &  \tilde{M}_3 \\
	0 & 0
	\end{pmatrix}\cos \theta 
	=   \begin{pmatrix} 
	 \tilde{M}_1 & \eta \tilde{M}_3  \\
	0 &  \tilde{M}_2
	\end{pmatrix}\sin\theta
	\end{split}
\end{equation}
where 
 \begin{equation}
 \label{eq:constM1}
 \begin{split}
	\tilde{M}_1&=\tilde{T}^{-1}M_1\tilde{T} \\
		&=\begin{pmatrix}
		\ddots & \vdots & \vdots & \vdots & \vdots & \vdots & \ddots \\
		\ldots & 0 & 1 & 0 & 0 & 0 &\ldots \\
		\ldots & -1  & 0 & 1 & 0 & 0 & \ldots \\
		\ldots &  0 &  -1 & 0 & 1 & 0  & \ldots \\
		\ldots & 0 & 0 &  -1  & 0 & 1 &  \ldots \\
		\ldots & 0 & 0 &  0 & -1 & 0 &   \ldots \\
		\ddots & \vdots & \vdots & \vdots & \vdots & \vdots & \ddots 
	\end{pmatrix},
	\end{split}
\end{equation}
 \begin{equation}
 \label{eq:reducedM3}
 \begin{split}
	\tilde{M}_3&=\tilde{T}^{-1}M_1\tilde{T}-\tilde{T}^{-1}M_3-M_2 \\
		&=\begin{pmatrix}
		\ddots & \vdots & \vdots & \vdots & \vdots & \vdots & \ddots \\
		\ldots & 0 & \frac{2}{\alpha_{-1}^2} & 0 & 0 & 0 &\ldots \\
		\ldots & 0  & 0 & \frac{2}{\alpha_{0}^2} & 0 & 0 & \ldots \\
		\ldots &  0 &  0 & 0 & \frac{2}{\alpha_{1}^2} & 0  & \ldots \\
		\ldots & 0 & 0 &  0  & 0 & \frac{2}{\alpha_{2}^2} &  \ldots \\
		\ldots & 0 & 0 &  0 & 0 & 0 &   \ldots \\
		\ddots & \vdots & \vdots & \vdots  & \vdots & \vdots & \ddots 
	\end{pmatrix},
	\end{split}
\end{equation}
$\tilde{M}_2=M_2$, and $\eta=\cot \theta$.

The matrix $M_2$ is the same as the matrix governing the dynamics of a class of the linearised equations on a two-dimensional domain, as in Section 2.5 of \cite{DMW16}. 
Thus the dynamics have been split into two parts; one driven by $M_2=\tilde{M}_2$ describing the dynamics of a two-dimensional linearised class, and another determined by  $\tilde{M}_1$ with constant coefficient coupling to the 2nd group of modes given by $\tilde{M}_3$. The coupling between the two sets of modes depends on the value of $\eta$. If $\eta=0$, there is no coupling and the matrix is block diagonal. We will now show that $\tilde{\mathcal{M}}$ is always similar to such a block diagonal matrix if $\eta\neq 0$.

Define $\rho_k$ as  
\begin{equation}
\label{eq:rhored}
	\rho_k:=1-\frac{1}{\alpha_n^2}
\end{equation} 
analogous to \cite{DMW16}.
Then $\tilde{M}_2$ can be written as
\begin{equation}
 \label{eq:rhoM2}
	\tilde{M}_2=\begin{pmatrix}
		\ddots & \vdots & \vdots & \vdots & \vdots & \vdots & \ddots \\
		\ldots & 0 & \rho_{-1} & 0 & 0 & 0 &\vdots \\
		\ldots & -\rho_{-2} & 0 & \rho_0 & 0 & 0 & \vdots \\
		\ldots &  0 &  -\rho_{-1} & 0 & \rho_1 & 0  & \vdots \\
		\ldots & 0 & 0 &  -\rho_0  & 0 & \rho_2 &  \vdots \\
		\ldots & 0 & 0 &  0 & -\rho_1 & 0 &   \vdots \\
		\ddots & \vdots & \vdots & \vdots & \vdots & \vdots & \ddots 
	\end{pmatrix}
\end{equation}
and $\tilde{M}_3$ can be written as 
 \begin{equation}
 \label{eq:rhoM3}
	\tilde{M}_3=\begin{pmatrix}
		\ddots & \vdots & \vdots & \vdots & \vdots & \vdots & \ddots \\
		\ldots & 0 & 2(1-\rho_{-1}) & 0 & 0 & 0 &\vdots \\
		\ldots & 0  & 0 & 2(1-\rho_{0}) & 0 & 0 & \vdots \\
		\ldots &  0 &  0 & 0 & 2(1-\rho_{1}) & 0  & \vdots \\
		\ldots & 0 & 0 &  0  & 0 & 2(1-\rho_{2}) &  \vdots \\
		\ldots & 0 & 0 &  0 & 0 & 0 &   \vdots \\
		\ddots & \vdots & \vdots & \vdots & \vdots & \vdots & \ddots 
	\end{pmatrix}
\end{equation}
Now define the matrix
\begin{equation}
	B=\begin{pmatrix}
		\ddots & \vdots & \vdots & \vdots & \vdots& \vdots & \vdots & \ddots \\
		\ldots & -1 & 0 & +1 & 0 & +1 & 0  &\vdots \\
		\ldots & 0 & -1 & 0 & +1 & 0 & +1   &\vdots \\
		\ldots & -1 & 0 & -1 & 0 & +1 & 0   &\vdots \\
		\ldots & 0 & -1 & 0  & -1 & 0 & +1 & \vdots \\
		\ldots & -1 & 0 & -1 & 0  & -1 & 0  & \vdots \\
		\ddots & \vdots & \vdots & \vdots & \vdots& \vdots & \vdots & \ddots 
	\end{pmatrix}.
\end{equation}

This is a Topelitz operator \cite{Gray06} composed of shifted copies of the row \begin{equation}(...,b_{-3},b_{-2},b_{-1},b_{0},b_1,b_2,b_3,...)\end{equation} where
\begin{equation}
	b_i=\begin{cases}
		+1 \text{ if } i=2k\text{ for some }k>0 \\
		-1 \text{ if } i=2k\text{ for some } k\leq 0 \\
		0 \text{ otherwise.}
	\end{cases}
\end{equation}

\begin{lemma}
\label{lem:laxpair}
The operator $\tilde{\mathcal{M}}(\eta)$ and the operator $\mathcal{B} =  \begin{pmatrix} 0 & B \\ 0 & 0 \end{pmatrix}$ are a Lax pair
\[
    \frac{d}{d\eta} \tilde{\mathcal{M}}(\eta) = [\tilde{\mathcal{M}}(\eta), \mathcal{B}]
\]
so that $\tilde{\mathcal{M}}(\eta)$ is an iso-spectral deformation of $\tilde{\mathcal{M}}(0)$.
\end{lemma}
\begin{proof}
The  operator $\tilde{\mathcal{M}}(\eta)$ decomposes into the block-diagonal operator  $X =\tilde{\mathcal{M}}(0)$, 
and the block upper triangular (and hence nilpotent) operator $N$ such that $\tilde{\mathcal{M}}(\eta) = X + \eta N$. 
Any two such nilpotent operators have vanishing commutator, so that the Lax equation becomes
\[
       N = [X + \eta N, \mathcal{B}] = [X, \mathcal{B}]
\]
which in block form reduces to 
\[
     \tilde M_3 = \tilde M_1 B - B \tilde M_2 \,.
\]
It is straightforward to verify that $B$ indeed satisfies this equation.
\end{proof}

The similarity between $\tilde{\mathcal{M}}(0)$ and $\tilde{\mathcal{M}}(\eta)$ is simply given by exponentiating $\mathcal{ B}$, so that
 $\tilde{\mathcal{M}}(0) = \exp(-\eta \mathcal{B}) \tilde{\mathcal{M}}(\eta)  \exp(\eta \mathcal{B}) $
where explicitly  
\begin{equation}
	\exp( \eta \mathcal{ B}) = \begin{pmatrix} \mathbb{I} & \eta {B} \\ 0 & \mathbb{I} \end{pmatrix}, \quad
	\exp( -\eta \mathcal{ B} ) =\begin{pmatrix} \mathbb{I} & -\eta {B} \\ 0 & \mathbb{I} \end{pmatrix}.
\end{equation}

The  operator $\tilde{\mathcal{M}}(0)$ is independent of $\eta$ and block diagonal with blocks $\tilde{M}_1$, $\tilde{M}_2$. The matrix $\tilde{M}_1$ is Toeplitz 
and thus diagonalisable by a discrete Fourier transform \cite{Karner02} and $\tilde{M}_2$ is diagonalisable by the results of \cite{DMW16}. We thus arrive at the following conclusion.

\begin{theorem}[Equating two- and three-dimensional classes]
\label{thm:2d23d}
Assume that $\Gb$, $\pb$, and $\ab$ linearly independent, and initial conditions are on the divergence-free subspace. 
Then the three-dimensional class with parameters $\ab$, $\pb$, $\Gb$ is linearly stable if and only if
 the two-dimensional class with parameters  $\tilde{\ab}=(\tilde{a}_x,\tilde{a}_y)$, $\tilde{\pb}=(1,0)$, $\Gamma=\tilde{a}_y\sin \theta$ is linearly stable, where the reduced parameters
 $\tilde{a}_x$, $\tilde{a}_y$ and $\theta$ are given in Lemma~\ref{lem:redpar}.
\end{theorem}
\begin{proof}
By Lemma~\ref{lem:redpar}, the dynamics of a class with parameters $\ab$, $\pb$ and $\Gb$ are isomorphic to the dynamics of a class with parameters $\tilde{\ab}=(\tilde{a}_x,\tilde{a}_y,0)$, $\pb=(1,0,0)$ and $\Gb=(0,\cos\theta,\sin\theta)$. 
By assumption $\Gb$, $\pb$, and $\ab$ linearly independent so that
$\Gb^\cdot (\ab\times\pb) \neq 0$ so $\theta \neq 0,\pi$. 

Then the dynamics of this class can be reduced by the divergence-free condition and are given by the matrix $\mathcal{M}$ \eqref{eq:3DM1}. If $\sin\theta\neq 0$, by the transformation in Section \ref{sec:laxred} the dynamics are given by the matrix $\tilde{\mathcal{M}}(0)$ . The matrix $\tilde{\mathcal{M}}(0)$ is block diagonal with a constant antisymmetric block $\tilde{{M}}_1$ \eqref{eq:constM1} and a block $\tilde{M}_2$. The spectrum of $\tilde{{M}}_1$ is purely imaginary as $i\tilde{{M}}_1$ is Hermitian, and covers the interval $2i[-1,1]$. 
The block $\tilde{M}_2$ is equal to the matrix giving the dynamics of the equivalent two-dimensional problem with the parameters $\ab=(\tilde{a}_x,\tilde{a}_y)$, $\pb=(1,0)$, and 
$\Gamma=\tilde{a}_y\sin\theta$; see e.g. equation (2.29) in \cite{DMW16} or equation (VI.1) in \cite{Li00} with the appropriate definition of $\rho_k$.
By \cite{Li04}, this also has continuous spectrum $2i[-1,1]$ 
and thus a value is in the spectrum of $\tilde{\mathcal{M}}$ if and only if it is in the spectrum of $\tilde{M}_2$. 
\end{proof}
This shows that the continous spectrum of a class is $2\tilde{a}_y|\sin\theta|i[-1,1]$. By considering all classes we can conclude from this that the continuous spectrum of the linearised Euler equations in a three-dimensional periodic domain is the full imaginary axis, analogous to the result in \cite{Li04}.

Note that even though the imaginary spectrum of the two- and three-dimensional classes are the same, the spectral density will be different, as there is a contribution to the imaginary spectrum from $\tilde{{M}}_1$ in addition to the contribution from $\tilde{{M}}_2$.

\section{Linearly stable Shear Flows}

\begin{figure}
\includegraphics[trim={0cm 0.3cm 0cm 0cm}, clip,width=0.49\textwidth]{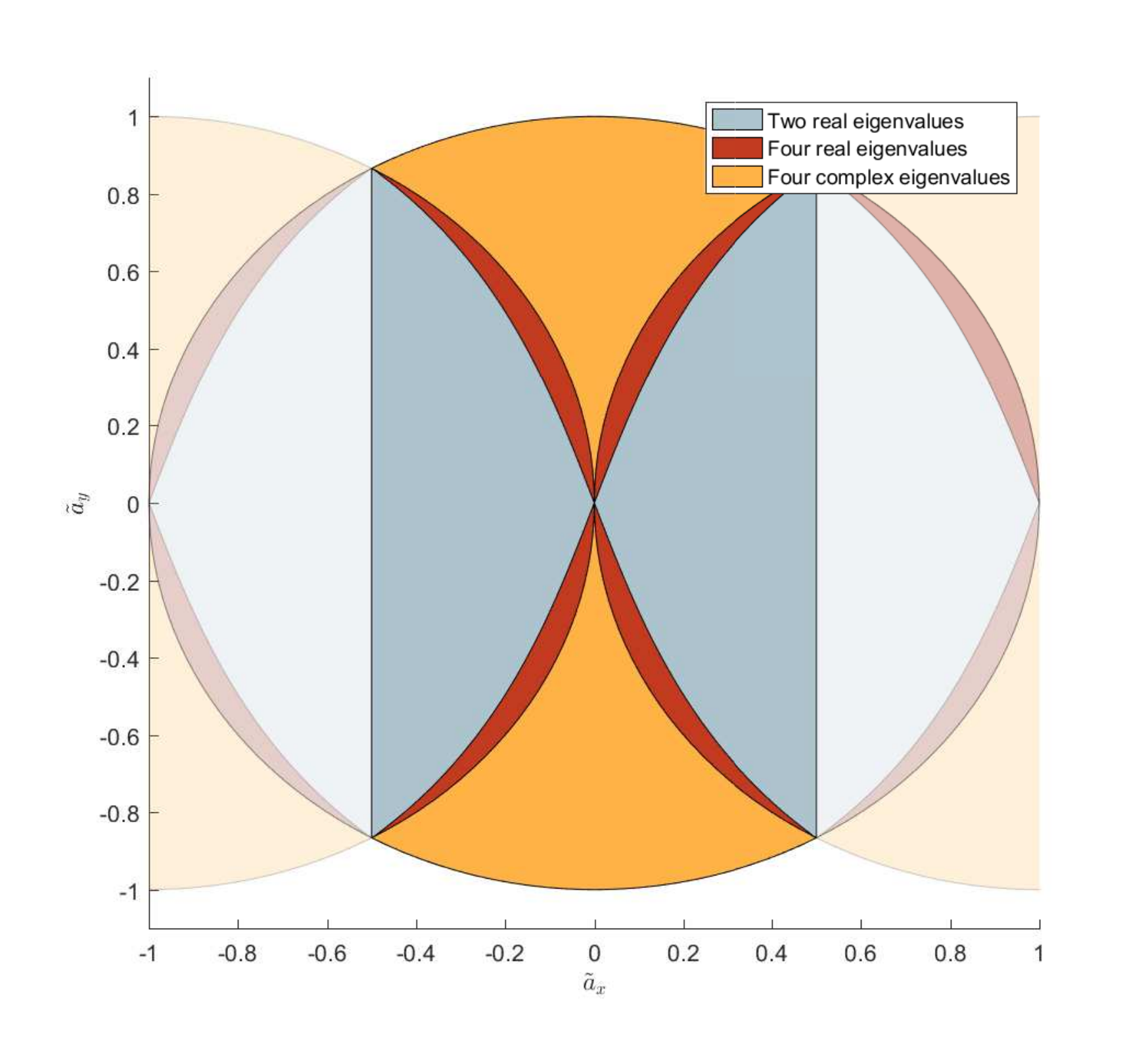}
\includegraphics[trim={0cm 2.5cm 0cm 0cm}, clip,width=0.49\textwidth]{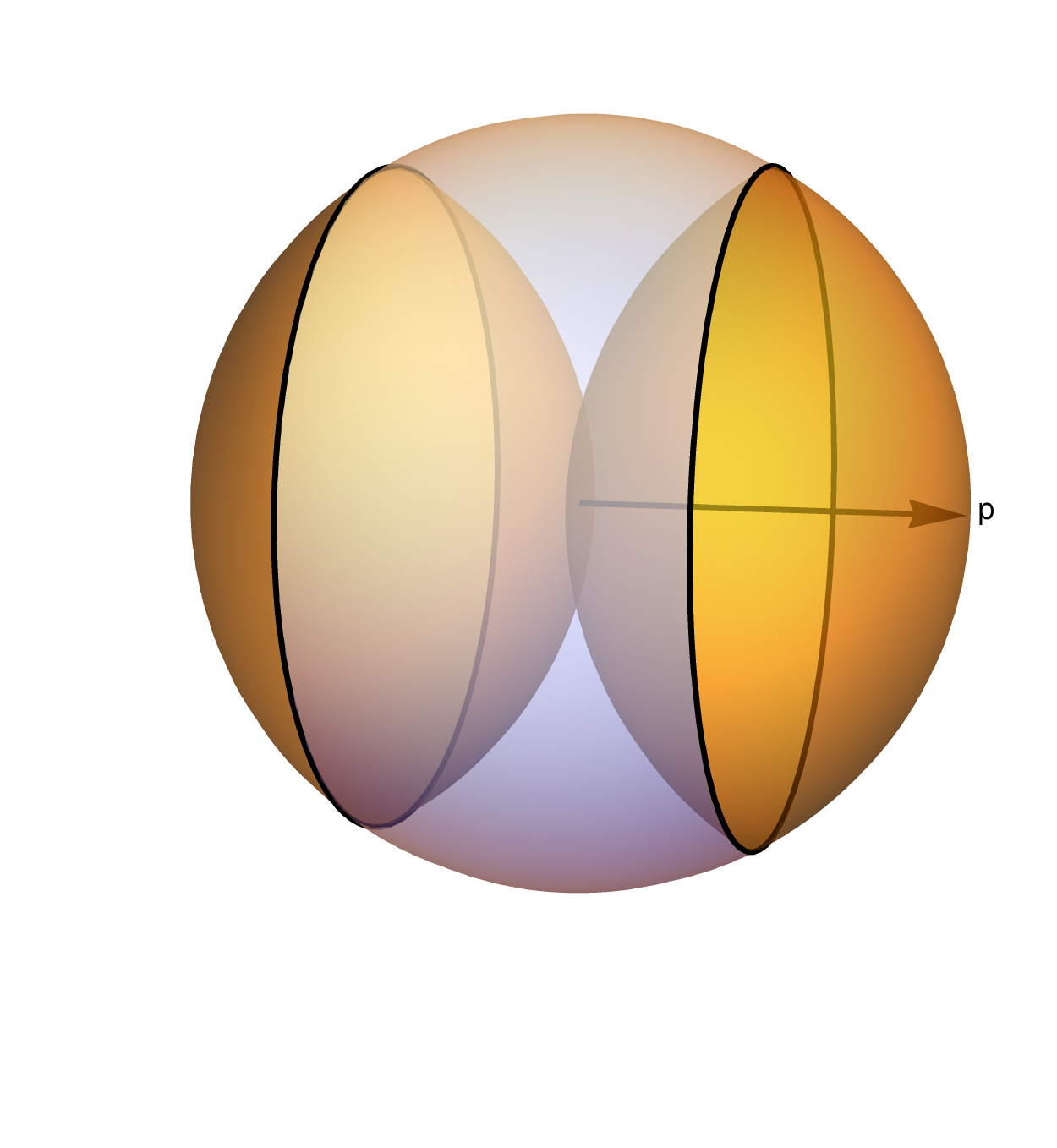}
\caption{Left: depending on  $\tilde{a}_x$, $\tilde{a}_y$, and $\theta\neq 0,\pi$ one can determine the number and type of nonimaginary eigenvalues of $\tilde{\mathcal{M}}$. If $\tilde{a}_x^2+\tilde{a}_y^2\geq 1$ there are no nonimaginary eigenvalues. If $(\tilde{a}_x\pm1)^2+\tilde{a}_y^2\geq 1$ there is a pair of real eigenvalues. Otherwise, there are four nonimaginary eigenvalues occurring as either two real pairs or a complex quadruplet. The lighter region is outside the principal domain $\mathcal{A}$. This is the unstable ellipse introduced in \cite{Li00} and analysed in \cite{DMW16}. Right: the unstable ellipsoid  as in Definition \ref{defn:uellipsoid}, which conveys the same information as the unstable ellipse in terms of the original parameter $\ab\in\mathbb{Z}^3$. The ellipsoid illustrated uses the parameters $\pb=(1,0,0)$ and $\aspb=(1,1,1)$; for other parameter values, the region will change by an affine transformation according to the scaling and rotations in Section \ref{sec:simplifying1}.}
\label{fig:ueAll}
\end{figure}

In Theorem \ref{thm:2d23d} it was shown that the dynamics of a three-dimensional class can be reduced to the dynamics of a two-dimensional class
when $\theta \not = 0, \pi$.
Thus we can obtain similar stability results. 
We begin by showing that all but finitely many of the classes do not contribute linear instability. By carefully controlling the domain size $\aspb$ and excluding $\theta=0,\pi$, we can then find shear flows that are linearly stable.

\subsection{Spectrally Stable Classes}

Define the \emph{unstable ellipsoid}.
\begin{definition}[The Unstable Ellipsoid]
\label{defn:uellipsoid}
 Define the unstable ellipsoid for $\pb\in \mathcal{L}$ as
 \begin{equation}
 \label{eq:unstableellipsoid}
 	D_\pb=\{ \xb\in\mathbb{R}^3 \; | \; |\xb|<|\pb| \}.
\end{equation}
\end{definition}

The unstable ellipsoid is illustrated in Figure \ref{fig:ueAll}. This object will serve the same purpose as the unstable disc introduced in \cite{Li00}, delineating linearly stable and unstable classes. Classes that do not intersect the unstable ellipsoid at a lattice point and do not satisfy $\theta =0,\pi$ cannot contribute linear instability.

Begin by noting that for general $\ab,\pb\in \mathbb{Z}^3$ with associated reduced parameters $\tilde{a}_x,\tilde{a}_y$ defined by \eqref{eq:3dparam}, $|\ab||\pb|=\sqrt{\tilde{a}_x^2+\tilde{a}_y^2}$. Thus the definition of $\rho_k$  is
\begin{equation}
\label{eq:rhofull}
\begin{split}
	\rho_k&:=1-\frac{1}{\tilde{a}_x^2+(\tilde{a}_y+k)^2} \\
		&={|\pb|^2}\left ( \frac{1}{|\pb|^2}-\frac{1}{|\ab|^2} \right ) .
\end{split}
\end{equation}
Then this matches the definition for $\rho_k$ in \cite{Dullin16} up to the constant factor of  ${|\pb|^2}$. 
Since $\tilde{M}_2$ is linear in $\rho_k$ 
this factor can be removed by a time rescaling. 

Also note that for $\ab \in \mathcal{A}$ the principal domain, $\ab \notin D_\pb \implies \ab+k\pb \notin D_\pb$ for all $k\in\mathbb{Z}$.

\begin{proposition}
\label{prop:insideunstable}
If and only if $\ab+k\pb \in D_\pb$, then $\rho_k<0$.
\end{proposition}
\begin{proof} 
By definition, $\ab+k\pb\in D_\pb\;\iff \;|\ab|<|\pb|$. Therefore, $\rho_k<0$ by \eqref{eq:rhofull}.
\end{proof}

\begin{proposition}[Stable Classes]
\label{prop:stableclasses3D}
If $\Gb\cdot(\ab\times\pb) \neq 0$ and $\ab\notin D_\pb$, the associated class is linearly stable. 
\end{proposition}
\begin{proof}
If $\Gb\cdot(\ab\times\pb) \neq 0$, then $\theta \neq 0,\pi$.
Thus by Theorem \ref{thm:2d23d}, the system is diagonalisable and the spectrum is that of the associated two-dimensional problem. If $\ab \notin D_\pb$, as $\ab$ is in the principal domain \eqref{eq:ADomain3D} then $\ab+k\pb\notin D_\pb$ for all $k\in\mathbb{Z}$. Thus by Proposition \ref{prop:insideunstable}, $\rho_k\geq 0$ for all $k\in\mathbb{Z}$. Thus by Section 3.1 of \cite{DMW16}, the class has only imaginary spectrum and is diagonalisable
and is therefore linearly stable.
\end{proof}

Thus only finitely many classes can contribute linear instability; those that intersect the unstable ellipsoid at one of its finite interior integer lattice points or satisfy $\theta=0,\pi$. We can therefore limit our search for instability to this finite collection of classes.

\subsection{Linearly Stable Flows}

We have shown that all but finitely many classes cannot contribute linear instability. The classes that can contribute instability are exactly those that are led by $\ab$ that lie in the unstable ellipsoid. As the shape of this ellipsoid is determined by the size of the domain $\aspb$, it may be possible for certain $\pb$ to make the unstable ellipsoid sufficiently small to ensure there are no classes that contribute spectral instability. This is analogous to the result of \cite{Dullin16}; however, the existence of the nilpotent classes discussed in Section \ref{sec:linclasses} means there is an additional condition on $\Gb,\aspb$ to ensure there is  linear stability.

\begin{theorem}[Linearly Stable Shear Flows]
\label{thm:stable3D}
If $\pb=(\aspect_x p_x,0,0)$ with $p_x\in\mathbb{Z}$, $\aspect_y,\aspect_z>\aspect_x|p_x|$, and $(\aspect_z\Gamma_y)/ (\aspect_y\Gamma_z)$ is not rational, the stationary solution
\begin{equation}
	\Omega^*=2\Gb \cos (\aspect_x p_x x)
\end{equation}
of the three-dimensional Euler equations \eqref{eq:omvrelation} on the torus, \eqref{eq:omvpde}, is linearly stable.
\end{theorem}
\begin{proof}
If $\ab=(\aspect_x{a}_x,\aspect_y{a}_y,\aspect_z a_z)$, and $a_z$ is nonzero, then 
\begin{equation}
	|\ab|^2\geq \aspect_z^2 a_z^2 > \aspect_z^2>\aspect_x^2 |p_x|^2=|\pb|^2.
\end{equation}
Thus if $a_z$ is nonzero, the class has $|\ab|>|\pb|$ and cannot contribute spectral instability by Proposition \ref{prop:stableclasses3D} (as $a_z\neq 0$ means $\ab$ cannot be parallel to $\pb=(\aspect_x p_x,0,0)$). As the conditions of Proposition  \ref{lem:lattrank} are satisfied, the class also does not contribute linear instability. Similarly, if ${a}_y$ is nonzero $|\ab|>|\pb|$ and the class cannot contribute linear instability. Thus the only classes that can contribute linear instability are of the form $\ab=(\aspect_x{a}_x,0,0)=\frac{\aspect_x{a}_x}{p_x}\pb$. But then the dynamics are trivially stable as discussed in Section \ref{sec:classdecomp} and cannot contribute instability. Thus there are no classes that contribute linear instability, and the shear flow is linearly stable.
\end{proof}

A corresponding result holds for shear flows with $\pb=(0,p_y,0)$ or $\pb=(0,0,p_z)$ and the equivalent conditions on $\aspb$. We will show in Section \ref{sec:parunstable} that although some steady states are linearly stable, in three dimensions no steady states are \emph{parametrically} stable.

As an aside, note that if we relax the condition that $\Gb\cdot(\ab\times\pb)  \neq 0$ for all $\ab\in\mathbb{Z}^3$, then there exist nilpotent classes with unstable linear dynamics  as in Section \ref{sec:linclasses}. However, all classes are still spectrally stable as all other classes are block diagonal and all blocks have only imaginary spectrum. Then there is spectral but not linear stability. 

\section{Linearly Unstable Shear Flows}

\subsection{Linear Growth Classes}
\label{sec:linclasses}

The matrix $\tilde{\mathcal{M}}$ is singular when $\theta=0,\pi$ as $\sin \theta = 0$. 
This is the case excluded in Theorem \ref{thm:2d23d}. In this special case the dynamics of the three-dimensional system is 
fundamentally different from the two-dimensional system. It may appear as if the condition $\theta = 0,\pi$ is very special;
however, we will see that in a precise sense every one of the steady states we are studying here contains a class 
arbitrarily close to this situation. Since the transformation $T$ \eqref{eq:TTran} is singular when $\theta = 0, \pi$, 
we go back one step and consider the matrix $\mathcal{M}$. 

\begin{theorem}
Assume that the vectors $\Gb$, $\pb$, and $\ab$ are linearly dependent.
Then the operator of the corresponding class in the linearised Euler equation $\mathcal{M}$ 
given by \eqref{eq:3DM1} is nilpotent and  $\exp( \mathcal{M} t)$ is unbounded.
\end{theorem}
\begin{proof}
By assumption of linear dependence  we have that $\Gb \cdot ( \pb \times \ab) = 0$.
According to Lemma~\ref{lem:redpar} this implies that the angle $\theta$ satisfies $\theta = 0$ or $\pi$.
In this case, $\mathcal{M}$ defined by \eqref{eq:3DM1} is nilpotent since it is block-upper triangular and hence $\mathcal{M}^2=0$. 
Thus the associated solutions to \eqref{eq:ombode} are linear in $t$ and given by
\begin{equation}
\label{eq:lingro}
	\omb(t)=\exp( \mathcal{M} t) \omb(0) = \left(\mathbb{I}+\tilde{a}_y\begin{pmatrix} 0 & M_3 \\ 0 & 0 \end{pmatrix}t \right)\omb(0).
\end{equation}
Thus the  dynamics of this class is linearly unstable. The linearised equation exhibits growth that is not caused by 
hyperbolic eigenvalues, but instead by a nilpotent operator. The corresponding growth of the solution in time is 
at least linear in $t$.
\end{proof}

An important question is whether, given $\pb$, $\Gb$, and $\aspm$, for any actual lattice points  $\ab\in\mathbb{Z}^3$ (rather than the reduced parameters $\tilde{a}_x, \tilde{a}_y\in\mathbb{R}$) the equation $\sin\theta=0$ holds. For now we consider the isotropic case. If $\ab= k\pb$ for some $k\in\mathbb{R}$, then the condition is automatically satisfied, but these classes only consist of constant modes by the argument in Section \ref{sec:classdecomp} and cannot contribute instability.  
Note that normalisation of parameters assumes that $\ab$ is given and hence $\theta$ is determined, 
while here we are asking the question whether for a given $\Gb$ there exists some $\ab$ such that $\theta = 0, \pi$.

We can think of this geometrically. Consider the plane through the origin that contains 
the direction $\pb$ and the direction $\Gb$. 
Does this plane contain lattice points other than those on the ray determined by $\pb$?
A normal vector $\mathbf{n}$ to the plane is given by $\mathbf{n} = \Gb \times \pb$,
and the equation for $\ab \in \Z^3$ to lie on that plane becomes $\mathbf{n} \cdot \ab = 0$.
In general the solution of a linear homogeneous diophantine equation is a lattice.
Since $\pb$ is an integer vector in the plane the rank of this lattice is at least one. 
The rank cannot be more than two since the lattice is contained in the plane.
When the rank is 2 then this implies that there are lattice points $\ab$ and hence
classes for which $\Gb$, $\pb$, $\ab$ are coplanar and hence $\theta = 0, \pi$.

\begin{lemma}
\label{lem:lattrank}
The rank of the lattice of solutions $\ab \in \Z^3$ of the plane 
$\mathbf{n} \cdot \ab = 0$ where $\mathbf{n} = \Gb \times \pb$, $\pb \in \Z^3$,
is 2 if and only if $\Gb$ is proportional to an integer vector.
\end{lemma}
\begin{proof}
The rank of the lattice is at least one since $\pb \in \Z^3$.
Clearly when $\Gb$ is proportional to an integer vector the rank of the lattice is 2,
since $\pb$ and $\Gb$ are linearly independent.
Conversely, choose an $\ab \in \Z^3$ and assume it to be on the plane.
Thus a normal to the plane is $\pb \times \ab$.
Then $\Gb$ has the unique direction which is orthogonal to $\pb$ 
(by the divergence free condition) and orthogonal to $\pb \times \ab$, 
and therefore the direction of $\Gb$ is given by $\pb \times (\pb \times \ab)$.
Thus we have shown that the rank is 2 if and only if $\Gb$ is proportional 
to an integer vector.
\end{proof}

It is easy to incorporate the anisotropic case in all of the results in this section, 
as usual by considering $\pb$ and $\ab$ in $\mathcal{L}$ rather than $\mathbb{Z}^3$.
Note that this means that instead of perturbing $\Gb$ to achieve $\theta = 0, \pi$ for 
some $\ab$, one can keep $\Gb$ fixed and perturb $\aspm$ to the same end.

The question of when $\theta = 0, \pi$ is possible is most interesting for vectors 
$\pb$ such that the linearisation about the corresponding steady state is at 
least spectrally stable, because then a nilpotent linearisation leads  to linear instability.
From Proposition \ref{prop:stableclasses3D} we know that this implies that 
no lattice point is inside the unstable ellipsoid (as there would be an unstable class otherwise), 
and this implies that up to permuting axis 
we have $\pb=(p_x,0,0)$. By the divergence free condition $\Gamma_x = 0$,
and so the condition of Theorem \ref{thm:stable3D} (incorporating the $\aspect$ factors) is
that the ratio $\aspect_z \Gamma_y : \aspect_y \Gamma_z$ is rational.
We will discuss an interpretation of these classes in Section \ref{sec:nearnil}.

When $\Gb$ is not proportional to an integer vector it is interesting to ask 
the question how close one can get to a lattice point (not on the line through $\pb$).
This is a problem in diophantine approximation. In the particular case where $\pb = (p_x, 0, 0)$ 
and $\Gb = \Gamma ( 0, \cos\psi, \sin\psi)$ and $\ab$ are the original un-rotated vectors
the formula determining $\theta$
simplifies to
\[
    \tan \theta = \frac{{a}_y \sin\psi - a_z \cos\psi}{{a}_y \cos\psi + a_z \sin\psi}
\]
or, as usual, when incorporating $\aspb$ the components of $\ab$ get multiplied by the scale factors.
When the numerator becomes small $\tan\theta$ is small, and hence we are trying to find 
an approximation of $\tan\psi$ by the rational number $\sfrac{a_z}{{a}_y}$. 
To measure how well a number $x=\tan\psi$ can be approximated by rationals 
define the limit
\[
     c(x) = \lim_{q \to \infty}  q^2 \left| x - \frac{p}{q} \right| \,,
\]
where $p$ is nearest integer to $q x$ with $q\in\mathbb{Z}$. A number is called badly approximable,
see, e.g., \cite{Burger00}, if $c(x)$ is non-zero.
In our case this can be rewritten as
\[
    c \cos\psi  \approx  {a}_y |{a}_y \sin\psi - a_z \cos\psi| 
\]
and assuming that $0 < \psi < \pi/2$ we find
\[
     \frac{c}{{a}_y^2} \cos^2\psi  \approx \frac{c }{{a}_y^2 } \frac{1}{ 1 + \tan\psi \frac{a_z}{{a}_y}}  \approx  | \tan\theta | = \frac{1}{\eta} \,.
\]
Below we will show that the growth rate of a nilpotent (or nearly nilpotent) operator is $\eta / ({a}_y^2 + a_z^2)$
so that combining the results we get that the growth rate is $1/c$.
For a badly approximable number  the asymptotic growth rate of nearly nilpotent operators is finite,
while for a well approximable number the growth rate diverges. 
This appears to indicate that the behaviour of the linearised operator does not only depend
on whether the slope of $\Gb$ given by $\tan\psi$ is rational or irrational, but that in the case 
that it is irrational the number-theoretic properties of $\tan\psi$ become important.

\subsection{Parametric Instability}
\label{sec:parunstable}

In order to emphasise the fact that linearly stable and linearly unstable steady states are arbitrarily close to each other 
we use a variant of concept of parametric instability of matrices. Recall that a matrix is called parametrically unstable if
it can be made unstable by an arbitrarily small perturbation. Since we are in infinite dimensions we restrict the 
possible perturbations to be within our family of steady states.
\begin{definition}
We call a member of a family of steady states \emph{parametrically unstable} if by an arbitrary small perturbation
within the family the steady state can be made linearly unstable.
\end{definition}
Sometimes the term \emph{strong stability} is used do designate the absence of parametric instability.

The last theorem shows that  cases of linear instability for fixed $\mathbf p$ are dense in the set of allowed $\Gb$.
This  observation leads to the following theorem:
\begin{theorem}
Any steady state of the form \eqref{eqn:OmegaStar} in the three-dimensional Euler equations on the torus is parametrically unstable.
\end{theorem}
\begin{proof}
Assume that the linearised equations only have continuous spectrum for the given values 
of $\mathbf p$, $\mathbf \Gamma$ and $\aspm$. Since $\mathbf p\in\mathcal{L}$ is discrete it is fixed.
In Lemma \ref{lem:lattrank} we showed that linearly unstable cases are dense in the space of all 
$\Gb$ and $K$. Thus, arbitrarily close to the given stable steady state there is 
a linearly unstable steady state. More precisely, for any given size of perturbation 
$\epsilon$ in the parameters, there are parameters for which some mode exhibits linear
growth.
\end{proof}

The mode number which corresponds to the growing operator may of course 
be very high depending on the size of the perturbation and the resulting $\frac{1}{c}$ growth rate. 
An interesting quantity to consider is  how the mode-number grows 
when sending the perturbation $\epsilon$ to zero.
The previous theorem explains that in a precise sense the linearly stable steady states are
in fact not stable in the sense of parametric instability. 
Another interpretation of the lurking instability in the problem is to realise that even 
linearly stable classes may lead to transient growth; this is explored in the following section.

\subsection{Nonnormality and Transition to Turbulence}

\label{sec:nearnil}

In Proposition \ref{prop:stableclasses3D} it was shown that if $\Gb\cdot(\ab\times\pb) \neq 0$ the class led by $\ab$ is not nilpotent and does not  contribute linear growth terms. Futhermore, in Proposition \ref{lem:lattrank} it was shown that there exist choices of the parameters $\pb,\Gb,\aspb$ such that for all  $\ab\in \mathcal{L}$ not collinear with $\pb$, $\Gb\cdot(\ab\times\pb)  \neq 0$.  However, if we consider all $\ab\in\mathcal{L}$, values of $\ab\in\mathcal{L}$ can be found such that $\Gb\cdot(\ab\times\pb) $ is arbitrarily small. At first sight this does not affect our analysis. We can still apply the transformation  as in Lemma \ref{lem:laxpair}, remove the influence of the $\mathcal{M}_3$ term, and make conclusions about spectral stability or instability. Nevertheless, there is an important point to be made here with implications for the dynamics of the full system. 

For small values of values of $\Gb\cdot(\ab\times\pb)  $ the parameter $\eta$ is correspondingly large. For all nonzero values of $\eta$, the matrix $\tilde{\mathcal{M}}$ \eqref{eq:goodmeq} is nonnormal, see, e.g.,  \cite{Trefethen05}, as the commutator of $\tilde{\mathcal{M}}$ and its conjugate transpose is
\begin{equation}
\begin{split}
	[\tilde{\mathcal{M}}&,\tilde{\mathcal{M}}^{\transp}] =\tilde{\mathcal{M}}\tilde{\mathcal{M}}^{\transp} - \tilde{\mathcal{M}}^{\transp}\tilde{\mathcal{M}}\\
	&=	\begin{pmatrix}
		\eta^2\tilde{M}_3\tilde{M}_3^{\transp} & \eta (\tilde{M}_3\tilde{M}_2^{\transp} -\tilde{M}_1^{\transp} \tilde{M}_3) \\
		\eta ( \tilde{M}_2\tilde{M}_3^{\transp}-\tilde{M}_3^{\transp}\tilde{M}_1) &
			(\tilde{M}_2^{\transp} \tilde{M}_2 - \tilde{M}_2 \tilde{M}_2^{\transp} )-\eta^2 \tilde{M}_3^{\transp} \tilde{M}_3 
\end{pmatrix}\sin^2 \theta\\ 
&\neq 0.
\end{split}
\end{equation}
One can think of $\tilde{\mathcal{M}}$ as becoming ``more nonnormal'' for larger values of $\eta$. The concept of a measure of nonnormality is discussed in \cite{Elsner87}. A nonnormal matrix may be linearly stable, but have dramatic transient behaviour that can cause instability in the full nonlinear system. In Boberg and Brosa \cite{Boberg88}, it is shown how turbulence can arise from shear flows in the Navier-Stokes equation in a pipe. They showed that small transient instabilities that can occur in the linearised problem can propagate into larger instabilities in the nonlinear problem, leading to turbulence. This instability is attributed to the degeneracy or near-degeneracy of eigenvalues; if two eigenvalues are very close, the transient dynamics on a short timescale may grow dramatically due to near-parallel eigenfunctions. In Trefethen et al \cite{Trefethen93} (see also \cite{Trefethen05}) it is shown that this behaviour can more generally be attributed to the nonnormality of the associated operator. This transition from laminar to turbulent flow is of much physical interest, and so has been well-studied \cite{Eckhardt07}.

Thus for classes that are spectrally stable as in  Proposition \ref{prop:stableclasses3D}, nonnormality can cause transient dynamics that are stable in the linearised problem, but are harbingers of nonlinear instability. Note that this occurs for sufficiently small values of $\Gb\cdot(\ab\times\pb) $; as $\ab$ must take all values in $\mathcal{L}=\aspm\mathbb{Z}^3$, this quantity can be made arbitrarily small by selecting appropriate $\ab$. Studying this in the context of 
nonlinear transition to turbulence of shear flows in a three-dimensional domain is a very interesting problem.

\section{Hyperbolic  Shear Flows}

In the previous sections we discussed linearly stable shear flows and spectrally stable shear flows that are linearly unstable.
The latter case is the truly three-dimensional novel case that has no analogue in two dimensions. 
Moreover, it implies that all linearly stable shear flows are arbitrarily close to linearly unstable cases in the sense of 
parametric instability. Non-linear stability results as in the two-dimensional case are therefore appear to be impossible 
in the three-dimensional case. As usual instability is structurally stable, and therefore results about  
instability with hyperbolic eigenvalues that hold in two dimensions do cary over to the three-dimensional case. 
Our analysis in the following is similar to \cite{DMW16}, but we expect that the more general and more rigorous results 
obtained in \cite{DLMVW19} would also carry over.

\subsection{Classifying Hyperbolic Eigenvalues}

\label{sec:nonim3D}

Having discussed the range of parameters $\ab$ for which there are no nonimaginary eigenvalues, we now look at the larger set of parameters for which we can identify nonimaginary parts. In many cases we can prove the existence of a real eigenvalue above a given lower bound. Based on numerical observations, we can classify the number and type of nonimaginary eigenvalues in these cases. Unless otherwise stated, we are in the anisotropic case with $\ab,\pb\in \mathcal{L}$. Assume that $\Gb\cdot(\ab\times\pb)  \neq 0$, so we are not in a nilpotent class with linear instability. Then we make the following conjectures based on numerical evidence: 
\begin{itemize}
	\item If $\ab\in D_\pb$ but $\ab+k\pb \notin D_\pb$ for all nonzero $k\in\mathbb{Z}$ the associated class has two real nonzero eigenvalues. 
	\item if $\ab\in D_\pb$, $\ab+\pb \in D_\pb$ or $\ab-\pb \in D_\pb$, but $\ab+k\pb\notin D_\pb$ for all other values of $k$, the associated class has either four real nonzero eigenvalues or a nonzero complex quadruplet of eigenvalues.
\end{itemize}
These observations are in direct analogy to the observations in \cite{DMW16,Dullin16,Dullin18} and are illustrated in Figure \ref{fig:ueAll}. Under particular conditions on $\pb$ we can now prove that there is a class with a positive real eigenvalue, leading to linear instability

\subsection{Unstable Shear Flows}

We now prove that, under a simple condition on $\pb$ we can find an $\ab$ such that the class led by $\ab$ has a positive real eigenvalue. This then implies linear instability.

\begin{theorem}[Unstable Shear Flows]
\label{thm:unsteady3D}
The steady state
\begin{equation}
	   \Omega^*=2\Gb \cos ( \pb\cdot\xb )
\end{equation}
is linearly unstable for all $\pb$ such that
\begin{equation}
	|\pb|>\sqrt{3}+\frac{3}{2}.
\end{equation}	
\end{theorem}
\begin{proof}
According to Theorem \ref{thm:2d23d}, the spectrum of the class with parameters $\ab$, $\pb$, $\Gb$ is equivalent to the spectrum of a two-dimensional class with parameters $\tilde{a}_x$, $\tilde{a}_y$ given by \eqref{eq:3dparam}. The parameters $\rho_k$ in this two-dimensional problem are then given by \eqref{eq:rhofull}. Then by Theorem 8 in  \cite{DMW16}, if $\rho_0<0$, $\rho_k>0$ for all $k\neq 0$, and $\rho_0+\rho_2<0$ (or equivalently $\rho_0+\rho_{-2}<0$) then there is a real eigenvalue $\lambda$ of $\tilde{M}_2$ and therefore $\tilde{\mathcal{M}}$ satisfying 
\begin{equation}
	\lambda \geq \lambda^*=\sqrt{-\rho_1 (\rho_0+\rho_2)}.
\end{equation}
Note that the proof in \cite{DMW16} relies on the $\rho_k$ values symbolically, so as long as the conditions of the theorem are satisfied the proof holds.

We now wish to prove there is a lattice point $\ab\in \mathcal{L}$ (where $\mathcal{L}$ is the anisotropic lattice defined in Section \ref{sec:aniso}) such that $\rho_0<0$, $\rho_k>0$ for all $k\neq 0$ and $\rho_0+\rho_2<0$.
Consider the ellipsoid
\begin{equation}
	E= \{ \xb\in\mathbb{R}^3 \;| \; 
		 |\xb-\cb|<\left (\frac{2}{\sqrt{3}}-1 \right )|\pb| \}
\end{equation}
where
\begin{equation}
	\cb=\frac{1}{\sqrt{3}}(\Gb\times  \pb) .
\end{equation}
Note that $\cb\cdot\pb=0$, and $|\cb|=\frac{1}{\sqrt{3}}|\pb|$ as $\Gb$, $\Gb\times \pb$ are perpendicular so $\cb$ and $ \pb$ are perpendicular.

It can be shown by an argument analogous to that appearing in \cite{DMW16} that if there is an integer lattice point $\ab\in E \cap \mathcal{L}$, $\rho_0<0$, $\rho_k>0$ for all $k\neq 0$ and $\rho_0+\rho_{\pm 2}<0$.

The ellipsoid $E$ contains a cube with side length  
\begin{equation}	
	d=2\left (\frac{2}{\sqrt{3}}-1\right ) |\pb|
\end{equation}
by a geometric argument.
If $d>1$, this cube is guaranteed to contain some integer lattice point, which we can select as our $\ab$. Thus a sufficient condition for such an $\ab$ to exist is
\begin{equation}
	|\pb|>\frac{1}{2\left (\frac{2}{\sqrt{3}}-1\right )} =\sqrt{3}+\frac{3}{2}.
\end{equation}
If such an $\ab$ exists, the $\rho_k$ values satisfy the conditions above, and so by Theorem \ref{thm:2d23d} and Theorem 8 in \cite{DMW16}, the associated class has a positive real eigenvalue with an explicit lower bound $\lambda^*=\sqrt{-\rho_1(\rho_0+\rho_2)}$. Then including the scaling factor due to the reduced parameter angle $\theta$, the linearised class has a real eigenvalue larger than $\tilde{a}_y|\sin \theta|\lambda^*$. For all $\ab\in E$, $\theta\neq 0,\pi$, so this is a positive lower bound.
Therefore the class has a positive real eigenvalue, and therefore is linearly unstable.
\end{proof}

This proof is based on the corresponding proof for the two-dimensional problem in Lemma 11 of \cite{DMW16}.

It is important to stress here that the condition $|\pb|>\sqrt{3}+\sfrac{3}{2}$ is sufficient, but not neccesary. In fact, numerical experiments suggests that all flows \emph{not} of the form described in Theorem \ref{thm:stable3D} will be unstable; see \cite{Worthington17} for more details on the numerics. 
However, for fixed domain size $\aspb$ Theorem \ref{thm:unsteady3D} proves instability for all but finitely many parameters $\pb$. 
Note that there is no required condition on $\Gb$, reflecting the idea that when the spectrum is
not purely imaginary $\Gb$ has limited influence on the stability of the shear flow.
Contrast this to the situation when the spectrum is purely imaginary, where exceptional 
values of $\Gb$ lead to linear instability.

	\section{Conclusion}
	
	The Euler equations on a three-dimensional periodic domain are more complicated and less well-understood than the two-dimensional equivalent. We have shown that the linearised system decomposes into classes, and that generically these classes have  dynamics equivalent to some corresponding class in the two-dimensional problem. This allowed us to prove that linearly 
	stable shear flows exist, while we also showed that most shear flows of the type considered are unstable with hyperbolic eigenvalues.
	There is a subset of the shear flows considered for which Theorem \ref{thm:unsteady3D} could not prove instability, but 
	from numerical experiment hyperbolic instability is clear for all $\pb$ except those covered by Theorem \ref{thm:stable3D}. 
	
	The most interesting results are related to the appearance of exceptional classes that only exist in three dimensions.
	They are spectrally stable but linearly unstable, with some growth of the solution in time $t$ that is at least linear in $t$.
	For certain special values of the parameters the corresponding shear flow will possess this type of instability.
	The fundamental importance of these exceptional classes is that they are in a precise sense arbitrarily close 
	to any shear flow in the family we considered. This lead to our main result that all shear flows of the family considered
	are parametrically unstable, i.e. arbitrarily close nearby within the family is linearly unstable shear flow.
	It is conceivable that such exceptional classes could lead to the parametric instability of all shear flows on 
	a periodic three-dimensional domain.
	Finally we discussed the importance of these exceptional classes for the possibility of a loss of nonlinear stability 
	triggered by the nonnormality of the corresponding linearised operator. 
	We hope that further research into the role of these exceptional classes in creating instability can shed some light 
	on the dynamics of the Euler equations in three-dimensions, particularly the transition from stability to turbulence.

\section*{Acknowledgements} 
The authors would like to thank Robert Marangell, Yuri Latushkin, and James Meiss for their suggestions and help.

\appendix

\section{Cross Product Matrix}

\label{app:crosshat}

Define the cross product matrix of a vector $\ab$
\begin{equation}
	\crossmat{\ab}:=\begin{pmatrix}
			0 & -a_z  & \tilde{a}_y \\
			a_z & 0 & -\tilde{a}_x \\
			-\tilde{a}_y & \tilde{a}_x & 0
			\end{pmatrix}
\end{equation}
where $\ab=(\tilde{a}_x,\tilde{a}_y,a_z)$. Then for any $\bb\in\mathbb{R}^3$, $\crossmat{\ab}\bb=\ab\times\bb$. Note that $\widehat{\ab}$ is antisymmetric due to the antisymmetry of the cross product.
For an invertible three by three matrix $M$, $M\ab\times M\bb=\text{det}(M)M^{-\transp}(\ab\times\bb)$, and therefore $\crossmat{M\ab}=\text{det}(M)M^{-\transp}\crossmat{\ab}M^{-1}$.  In the special case of a rotation matrix $R\in SO(3)$, $\crossmat{R\ab}=R\crossmat{\ab}R^{\transp}$.

\bibliographystyle{spmpsci}      
\bibliography{stability3D}{}   

\end{document}